\documentclass[generic,preprint,10pt]{imsart}
\RequirePackage{natbib}

\startlocaldefs

\usepackage{amssymb,amsmath,amsfonts,amsthm,amscd,amstext,mathtools}
\usepackage[english]{babel}
\usepackage[utf8]{inputenc}
\usepackage{enumerate}
\usepackage{graphicx}
\usepackage{url}
\usepackage[usenames,dvipsnames,svgnames,table]{xcolor}
\usepackage{caption}
\usepackage{subcaption}

\usepackage[a4paper,headsep=8mm,left=30mm,top=30mm,right=30mm,bottom=30mm]{geometry}

\usepackage[pdftex]{hyperref}

\numberwithin{equation}{section}

\newtheorem{theorem}{Theorem}[section]
\newtheorem{claim}[theorem]{Claim}
\newtheorem{lemma}[theorem]{Lemma}
\newtheorem{proposition}[theorem]{Proposition}
\newtheorem{corollary}[theorem]{Corollary}
\newtheorem{remark}[theorem]{Remark}
\newtheorem{definition}[theorem]{Definition}

\long\def\xcom#1{}

\newcommand{\cA}{{\ensuremath{\mathcal A}} }
\newcommand{\cB}{{\ensuremath{\mathcal B}} }

\newcommand{\cD}{{\ensuremath{\mathcal D}} }
\newcommand{\cE}{{\ensuremath{\mathcal E}} }
\newcommand{\cF}{{\ensuremath{\mathcal F}} }
\newcommand{\cG}{{\ensuremath{\mathcal G}} }
\newcommand{\cH}{{\ensuremath{\mathcal H}} }
\newcommand{\cI}{{\ensuremath{\mathcal I}} }

\newcommand{\cL}{{\ensuremath{\mathcal L}} }

\newcommand{\cO}{{\ensuremath{\mathcal O}} }

\newcommand{\cV}{{\ensuremath{\mathcal V}} }
\newcommand{\cW}{{\ensuremath{\mathcal W}} }

\newcommand{\gep}{\varepsilon}       % \ge already exists...

\renewcommand{\tilde}{\widetilde}          % wider `tilde'
\DeclareMathSymbol{\leqslant}{\mathalpha}{AMSa}{"36} % nicer `smaller or equal'
\DeclareMathSymbol{\geqslant}{\mathalpha}{AMSa}{"3E} % nicer `larger or equal'
\DeclareMathSymbol{\eset}{\mathalpha}{AMSb}{"3F}     % nicer `emptyset'
\newcommand{\dd}{\text{\rm d}}             % a straight d for differentials

\newcommand{\R}{\mathbb{R}}

\newcommand{\Z}{\mathbb{Z}}
\newcommand{\N}{\mathbb{N}}

\def\bs{\boldsymbol}

\newcommand\bP{\ensuremath{\bs{\mathrm{P}}}}
\newcommand\bE{\ensuremath{\bs{\mathrm{E}}}}

\newcommand{\ind}{{\sf 1}}

\renewcommand{\epsilon}{\varepsilon}
\renewcommand{\theta}{\vartheta}
\renewcommand{\phi}{\varphi}

 \newcommand{\be}[1]{\begin{equation}\label{#1}}
 \newcommand{\ee}{\end{equation}}

 \newcommand{\bl}[1]{\begin{lemma}\label{#1}}
 \newcommand{\el}{\end{lemma}}

 \newcommand{\br}[1]{\begin{remark}\label{#1}}
 \newcommand{\er}{\end{remark}}

 \newcommand{\bt}[1]{\begin{theorem}\label{#1}}
 \newcommand{\et}{\end{theorem}}

 \newcommand{\bd}[1]{\begin{definition}\label{#1}}
 \newcommand{\ed}{\end{definition}}

 \newcommand{\bcl}[1]{\begin{claim}\label{#1}}
 \newcommand{\ecl}{\end{claim}}

 \newcommand{\bp}[1]{\begin{proposition}\label{#1}}
 \newcommand{\ep}{\end{proposition}}

 \newcommand{\bc}[1]{\begin{corollary}\label{#1}}
 \newcommand{\ec}{\end{corollary}}

 \newcommand{\bpr}{\begin{proof}}
 \newcommand{\epr}{\end{proof}}

 \newcommand{\bi}{\begin{itemize}}
 \newcommand{\ei}{\end{itemize}}

\endlocaldefs

\definecolor{BL}{rgb}{0.22,0.45,0.70}

\begin{document}

\begin{frontmatter}
\title{Scaling limit of the uniform prudent walk}

\runtitle{Uniform prudent walk : scaling limits}

\begin{abstract}
We study the 2-dimensional uniform prudent self-avoiding walk, which assigns equal probability to all nearest-neighbor
self-avoiding paths of a fixed length that respect the prudent condition, namely, the path cannot take any step in the direction
of a previously visited site. The uniform prudent walk has been investigated with combinatorial techniques in \cite{B10}, while another
variant, the kinetic prudent walk has been analyzed in detail in \cite {BFV10}. In this paper, we prove that the
$2$-dimensional uniform prudent walk is ballistic and follows one of the $4$ diagonals with equal probability. We also establish a functional central limit theorem for the fluctuations of the path around the diagonal.
% after rescaling space and time by its length.
%and the interacting self-avoiding walk (ISAW)  for which the collapse transition was conjectured in \cite{S86}.
\end{abstract}

\author{\fnms{Nicolas} \snm{Pétrélis} \ead[label=e2]{nicolas.petrelis@univ-nantes.fr}}
\affiliation{Université de Nantes}

%\and
\author{\fnms{Rongfeng} \snm{Sun} \ead[label=e3]{matsr@nus.edu.sg}} \hspace{-4mm} \thanks{ R. Sun is supported by NUS grant R-146-000-220-112.}
\affiliation{NUS}

%\and
\author{\fnms{Niccolò}
  \snm{Torri}\corref{}\ead[label=e1]{niccolo.torri@univ-nantes.fr}
\ead[label=u1,url]{http://www.math.sciences.univ-nantes.fr/~torri/}}\hspace{-2mm} \thanks{N. Torri is supported by the \emph{``Investissements d'avenir"} program (ANR-11-LABX-0020-01).}
\affiliation{Université de Nantes}

\address{Laboratoire de Math\'ematiques Jean Leray UMR 6629\\
Universit\'e de Nantes, 2 Rue de la Houssini\`ere\\
BP 92208, F-44322 Nantes Cedex 03, France\\ \printead{e1}\\\printead{e2}
}

\address{National University of Singapore, NUS\\
 \printead{e3}
}

\runauthor{N. Pétrélis, R. Sun and  N. Torri}



\end{frontmatter}

%%%%%%%%%%%%%%%%%%%%%%%%%%%%%%%%%%%%%%%%%%%%%%%%%%%%%%%%%%%%%%%%%%%%%%%%%%%%%%
%%%%%%%%%%%%%%%%%%%%% Here the document begins %%%%%%%%%%%%%%%%%%%%%%%%%%%%%%%
%%%%%%%%%%%%%%%%%%%%%%%%%%%%%%%%%%%%%%%%%%%%%%%%%%%%%%%%%%%%%%%%%%%%%%%%%%%%%%

\section{Introduction}
The prudent walk was introduced in \cite{TD87b,TD87a} and \cite{SSK01} as a simplified version of the self-avoiding walk. It has attracted the attention of the combinatorics community in recent years, see e.g., \cite{B10,BI15,DG08}, and also the probability
community, see e.g. \cite{BFV10} and  \cite{PT16} .
\smallskip

In dimension $2$, for a given $L\in \N$, the set $\Omega_L$ of $L$-step prudent path on $\Z^2$ contains all nearest-neighbor self-avoiding path starting from the origin, which never take any step in the direction of a site already visited, i.e.,
\begin{align}\label{defPP}
\nonumber \Omega_L:=\big\{(\pi_i)_{i=0}^L\in (\Z^2)^{L+1}\colon\, &\pi_0=(0,0),  \pi_{i+1}-\pi_i\in \{\leftarrow,\rightarrow,\downarrow,\uparrow\}  \quad \forall  i\in \{0,\dots, L-1\}, \\
&  \big(\pi_i+\N (\pi_{i+1}-\pi_i)\big) \cap \pi_{[0,i]}=\emptyset \quad \,  \forall  i\in \{0,\dots, L-1\}\big\}
\end{align}
where $\pi_{[0,i]}$ is the range of $\pi$ at time $i$, i.e., $\pi_{[0,i]}=\{\pi_j: 0\leq j\leq i\}$.
\smallskip

Two natural laws can be considered on $\Omega_L$:
\begin{enumerate}
\item  The \emph{uniform} law  $\bP_{\text{unif},L}$, also referred to as the uniform prudent walk, under which at every path in $\Omega_L$ is assigned equal probability $1/|\Omega_L|$;

\item  The \emph{kinetic} law $\bP_{\text{kin},L}$, also referred to as the kinetic prudent walk, under which each step of the path is chosen uniformly among all the admissible steps. Note that the first step is in one of the $4$ directions with equal probability. Subsequently,
    if a step increases either the width or the height of its range, then the next step has $3$ admissible choices; otherwise there are only $2$ admissible choices. Let $\cH(\pi_{[0, L-1]})$ and $\cW(\pi_{[0, L-1]})$ denote the height and width of the range of $\pi_{[0, L-1]}$. Then, for $L\in \N$
and $\pi \in \Omega_L$, we note that
\begin{align}\label{linkpruki}
\bP_{\text{kin},L}(\pi)&=\tfrac{1}{4}\big(\tfrac{1}{2}\big)^{L-\cH(\pi_{[0, L-1]})-\cW(\pi_{[0, L-1]})}  \, \big(\tfrac{1}{3}\big)^{\cH(\pi_{[0, L-1]})+\cW(\pi_{[0, L-1]})}.
\end{align}
\end{enumerate}

\cite{BFV10} proved that the scaling limit of the kinetic prudent walk is given by
$Z_{u}=\int_{0}^{3u/7}\big( \sigma_1 \ind_{\{W_s\geq 0\}} {1 \choose 0} +\sigma_2 \ind_{\{W_s < 0\}} {0 \choose 1}\big)\dd s$, where $W$ is a Brownian motion and $\sigma_1,\sigma_2\in \{-1,1\}$ are random signs (independent of $W$), cf. \cite[Theorem 1]{BFV10}. %In particular, the asymptotic speed of the walk is well definined in

%\begin{theorem}[Theorem 1 \cite{BFV10}]
%On a suitably probability space we can construct the kinetic prudent walk $\pi$, a Brownian motion $W$ and a pair of $\pm 1$ random variables $\sigma_1,\sigma_2$ such that
%$$
%\lim_{t\to\infty}\bP\bigg( \sup\limits_{0\leq s \leq t} \big\Vert  \frac{1}{t}\pi_s-Z_{s/t}^{\sigma_1,\sigma_2}\big\Vert
%> \epsilon\bigg)=0,
%$$
%where, for $u\in [0,1]$,
%$$
%Z_{u}^{\sigma_1,\sigma_2}:=\int_{0}^{3u/7}\big( \sigma_1 \ind_{\{W_s\geq 0\}} +\sigma_2 \ind_{\{W_s < 0\}}\big)\dd s .
%$$
%The Brownian motion $W$ and the two random signs $\sigma_1,\sigma_2$ are independent of each other and $\bP(\sigma_1=s,\sigma_2=s')=1/$, for $s,s'\in\{-1,1\}$.
%\end{theorem}

\smallskip

In this paper, we identify rigorously the scaling limit of the 2-dimensional uniform prudent walk, proving a conjecture
raised in several papers, e.g., \cite[Section 5]{BFV10}, and \cite[Proposition 8]{B10} where partial answers were provided
for the \emph{2-sided} and \emph{3-sided} versions of the 2-dimensional prudent walk using combinatorial techniques. The conjecture, supported by numerical simulations, was that when space and time are rescaled by the length $L$, the 2-dimensional uniform prudent walk converges to a straight line in one of the 4 diagonal directions chosen with equal probability. This is in stark contrast to the kinetic prudent walk.
\smallskip
%To be more specific,  once rescaled in time and space by its total length $L$, a typical path sampled from $P_{\text{unif},L}$ was expected to stick to the straight line along the diagonal % chosen uniformly among the $4$ quadrants and it is conjectured that also the generic prudent walk displays the same scaling limit.
%This conjecture is somehow supported by Theorem \ref{thh1}.

%In a related work, \cite{PT16}  studied the 2-dimensional interacting prudent self-avoiding walk, generalizing results for the interacting partially directed self-avoiding walk (IPDSAW) in \cite{CGP13, CP15}.

%\subsection{Model}

\section{Main results}
\begin{definition}\label{def-pi-tilde}
For every $\pi\in \Omega_L$, let $\tilde \pi^L:[0,1]\mapsto \R^2$ be the rescaled and interpolated version of $\pi$, i.e.,
$$\tilde \pi^L_t=\frac{1}{L} \big(\pi_{\lfloor t L\rfloor}+(tL-\lfloor tL\rfloor) (\pi_{\lfloor tL\rfloor+1}-\pi_{\lfloor t L\rfloor})\big), \quad t\in [0,1].$$
\end{definition}
We also denote  $\vec{e}_1:=(1,1)$,  $\vec{e}_2:=(-1,1)$,  $\vec{e}_3:=(-1,-1)$ and  $\vec{e}_4:=(1,-1)$.

\smallskip

Our first result shows that the scaling limit of the uniform prudent walk is a straight line segment.

\begin{theorem}[Concentration along the diagonals]\label{unifscal}
%There exists a $c>0$ such that for every $i\in \{1,2,3,4\}$ and every $\gep > 0$,
There exists a $c>0$ such that for every $\gep >0$
\be{concentr}
\lim_{L\to \infty} \bP_{{\rm unif},L}\bigg(\exists\, i \in \{1,\dots ,4\} \ s.t.\  \sup_{t\in [0,1]}\big|\tilde \pi^L_t-c t\, \vec{e}_i \big|\leq \gep \bigg)=1. %\frac{1}{4}
\ee
\end{theorem}

Furthermore, we can identify the fluctuation of the prudent walk around the diagonal. More precisely, let $\sigma_L=1, 2, 3, 4$, depending
on whether $\tilde \pi^L_1$ lies in the interior of the 1st, 2nd, 3rd, or the 4th quadrant, and let $\sigma_L=0$ otherwise. Then we have

\begin{theorem}[Fluctuations around the diagonal]\label{unifscal2}
Under $\bP_{{\rm unif},L}$, the law of $\sigma_L$ converges to the uniform distribution on $\{1, 2, 3, 4\}$, and
\be{concentrclt}
\big(\sqrt L(\tilde \pi^L_t-ct \vec{e}_{\sigma_L})\big)_{t\in [0,1]} \Rightarrow (B_t)_{t\in [0,1]} \qquad \mbox{as } L\to\infty,
\ee
where $\Rightarrow$ denotes weak convergence, and $(B_t)_{t\geq 0}$ is a two-dimensional Brownian motion with a non-degenerate covariance matrix, cf.\ \eqref{covar}.
\end{theorem}

The proof of Theorem \ref{unifscal} follows the strategy used by \cite{BFV10}.
We consider the so called \emph{uniform 2-sided prudent walk} (cf. Section \ref{tspw}), a sub-family of prudent walks with a fixed diagonal direction.
%toward a corner of the plain (i.e., with a north-east, south-east, north-west, north-west orientation).
First we prove that the scaling limit of the uniform 2-sided prudent walk is a straight line, cf. Theorem \ref{scalingtwosided}. 
A weaker version of this result
was already proven by \cite[Proposition 6]{B10}. We reinforce it by using an alternative probabilistic approach.
We decompose a path into a sequence of excursions, which leads to an \emph{effective} one-dimensional random walk with geometrical increments, see e.g., Figure \ref{fig1}.
Then we show that under the uniform measure, a typical path of length $L$ crosses its range from one end to the other at most $\log L$ times
and the total length of the first $\log L$ excursions also grows at most logarithmically in $L$. This results refines the upper bound obtained by \cite{PT16}. The excursions crossing the range of the walk disappear in the scaling limit, while the remaining part of the path is nothing but a uniform 2-sided prudent walk (in one of the four diagonal directions), for which we have identified the correct scaling limit.
\smallskip

Theorem \ref{unifscal2} can be proved using the same strategy. Once it is shown to hold for the 2-sided uniform prudent walk, cf. Theorem \ref{scalingtwosided2}, then it also holds for the uniform prudent walk thanks to control on the number of excursions crossing the range of the walk.

\subsection{Organization of the paper}
The article is organized as follows: In Section \ref{tspw}, we introduce the uniform 2-sided prudent walk and identify its scaling limit.
%Then we show how decomposing a 2-sided path into a sequence of excursions and we show the nexus between such path and the effective random walk.
In Section \ref{pw}, we analyze the uniform prudent walk and prove some technical results needed to control the excursions crossing
the range of the walk. Lastly, we prove our main results Theorems \ref{unifscal} and \ref{unifscal2}  in Section \ref{sec5}.

\section{Uniform 2-sided prudent walk}\label{tspw}

Let $\Omega_L^{+}$ be the subset of $\Omega_L$ containing the so called \emph{2-sided} prudent path (in the north-east direction), that is, those paths $\pi\in \Omega_L$  satisfying three additional geometric constraints:
\begin{enumerate}
\item $\pi$ can not take any step in the direction of any site in the quadrant $(-\infty, 0]^2$;
%
%toward the south-west quadrant, i.e.,
%$%\text{QUAD}=
%(-\infty,0]^2$. To be more precise, let $\pi_{[0,i]}$ be the range of the walk at time $i$, and let
% $\vec{d}_1:=(1,0)$,  $\vec{d}_2:=(0,1)$,  $\vec{d}_3:=(-1,0)$ and  $\vec{d}_4:=(0,-1)$.
%Then the $(i+1)$-th increment $\vec{d}_j$
% with $j\in \{1,\dots,4\}$, is allowed only if
%$\{\pi_i+ t\vec{d}_j \colon t>0\}\cap (\pi_{[0,i]}\cup (-\infty,0]^2)=\emptyset$.
\item The endpoint $\pi_L$ is located at the top-right corner of the smallest rectangle containing $\pi$;
\item $\pi$ starts
with an east step ($\rightarrow$), i.e., $\pi_1=(1,0)$.
\end{enumerate}
We denote by
%$ \bP_{\text{unif},L}^+$ the uniform measure on $\Omega_L^+$
$ \bP_{\text{unif},L}^+$ the uniform measure on $\Omega_L^+$.
Theorems \ref{scalingtwosided} and \ref{scalingtwosided2} below are the counterparts of Theorems \ref{unifscal} and \ref{unifscal2} for the uniform 2-sided prudent walk. Recall that $\vec e_1=(1,1)$.

\begin{theorem}
\label{scalingtwosided}
%Let $ \bP_{\text{unif},L}^+$ be the uniform measure on the set $\Omega_L^+$ of the $2$-sided prudent paths.
There exists a $c>0$ such that for every $\gep > 0$,
\be{concentr2sided}
\lim_{L\to \infty} \bP_{{\rm unif},L}^+\bigg(\sup_{t\in [0,1]}\big|\tilde \pi^L_t-c t \, \vec{e}_1 \big|\le \gep \bigg)=1.
\ee
\end{theorem}

\begin{theorem}
\label{scalingtwosided2}
Under $\bP^+_{{\rm unif},L}$,
\be{concentr2sidedclt}
\big(\sqrt L(\tilde \pi^L_t-ct \vec{e}_1)\big)_{t\in [0,1]} \Rightarrow (B_t)_{t\in [0,1]} \qquad \mbox{as } L\to\infty.
\ee
where $B$ is the same two-dimensional Brownian motion as in Theorem \ref{unifscal2}.
\end{theorem}

\subsection{Decomposition of a 2-sided prudent path into excursions}\label{decexcu}

Every path $\pi\in \Omega_L^+$ can be decomposed in a unique manner into a sequence of horizontal and vertical excursions (see Figure \ref{fig1}).
First we introduce some notation. For $\pi\in \Omega_L^+$ and $i\leq L$, denote $\pi_i=(\pi_{i,1}, \pi_{i, 2})$.
Let $\tau_0:=0$ and 
\begin{align}
&\tau_1(\pi) :=\min\{ i>0 \, :\, \pi_{i,2}>0\}-1,
&\tau_2(\pi) :=\min\{ i>\tau_1 \, :\, \pi_{i,1}>\pi_{ \tau_{1,1}}\}-1,
\end{align}
which are the times when the first horizontal, resp.\, vertical excursion ends.
%be the index of the endpoint of the first horizontal and vertical excursion, respectively.
%Note that, by our assumptions, $\tau_1(\pi)$ is
%also the first time at which $\pi_{i,2}\neq 0$.
For $k\in \N$, define
\begin{align*}
&\tau_{2k+1}(\pi) :=\inf\{ i>\tau_{2k}\, :\, \pi_{i,2}>\pi_{ \tau_{2k,2}}\}-1,
& \tau_{2k+2}(\pi) :=\inf\{ i>\tau_{2k+1}\, :\, \pi_{i,1}>\pi_{ \tau_{2k+1,1}}\}-1.
\end{align*}
Let $\gamma_L(\pi):=\min\{j\geq 1\colon \, \tau_j(\pi) =\infty \}$ be the number of excursions in $\pi$. Note that each horizontal excursion starts with an east step, and each vertical excursion a north step. Since the endpoint $\pi_L$ lies at the top-right corner of the smallest rectangle containing $\pi$, the last excursion of $\pi$ can be made complete by adding an extra north step if it is a horizontal excursion, or adding an extra east step if it is a vertical excursion. Therefore, with a slight abuse of notation, we redefine $\tau_{\gamma_L}:=L$. We can thus decompose $\pi$ into the excursions $\big((\pi_{\tau_{k-1}},\dots,\pi_{\tau_{k}})\big)_{k=1}^{\gamma_L}$, which are horizontal for odd $k$ and vertical for even $k$.
% {\bf see figure...}.

\subsection{Effective random walk excursion} \label{sec:ERW}

Let $\cI_t$ denote the set of horizontal excursions of length $t$, flipped above the $x$-axis, i.e.,
\be{defIt}
\cI_t:=\big\{\pi=(\pi_0,\pi_1,\dots,\pi_t)\colon \pi_0=(0,0),\, \pi_1=(1,0),\ \pi_{i, 2} \ge 0\  \forall\, i \in\{1,\dots, t\},\, \pi_{t,2}= 0 \big\}.
\ee
Recall from Section \ref{decexcu} that each path  $\pi\in \Omega_L^+$ can be decomposed uniquely into $\gamma_L(\pi)$ excursions of length $\tau_i-\tau_{i-1}$, $i=1,\dots,\tau_L(\pi)$. These excursions are alternatingly horizontal and vertical, with the first excursion being horizontal, see Figure \ref{fig1}. We can thus partition $\Omega_L^+$ according to the value of $r:=\gamma_L(\pi)$ and the excursion lengths $t_1,\dots,t_r$. 
Defining 
\begin{equation}\label{defKbis}
K(t):=\frac{1}{2^t}\big | \cI_{t}\big |,
\end{equation}
we have that
\begin{equation}\label{two-sided1}
\frac{1}{2^{L}}\, |\Omega_L^{+}|=\sum_{r\geq 1} \sum_{t_1+\dots+t_r=L} \prod_{i=1}^r \Big | \cI_{t_i}\Big |\,  \frac{1}{2^{t_i}}
=\sum_{r\geq 1} \sum_{t_1+\dots+t_r=L} \prod_{i=1}^r K(t_i).
\end{equation}

We now follow the idea introduced in \cite{BFV10} and rewrite \eqref{defKbis} in terms of a one-dimensional \emph{effective random walk} $V=(V_i)_{i=0}^\infty$. The walk $V$ starts from $0$, has law  $\mathbf{P}$, and its increments
$(U_i)_{i=0}^\infty$ are i.i.d. and follow a discrete Laplace distribution, i.e.,
\begin{equation}\label{lawP}
\mathbf{P}(U_1 = x)=\frac{1}{3}\,  \frac{1}{2^{|x|}}, \quad x\in \Z.
\end{equation}

\begin{lemma}
Given the walk $V$ and $t\in \N$, let $\eta_t:=\min\big\{i\geq 1\colon \, i+\sum_{j=1}^{i} |U_j| \geq t\big\}$, then
\be{defK}
K(t)=\mathbf{E}\Big[e^{\log(\frac{3}{2})\, \eta_t} \,  \ind_{\{V_i\geq 0\ \forall i\leq \eta_t, \, V_{\eta_t=0}, \ \eta_t+\sum_{j=1}^{\eta_t} |U_j|=t\}}\Big].
\ee
\end{lemma}
\begin{proof}
For each $\pi\in \cI_t$ (cf. \eqref{defIt}), let $n(\pi):=|\pi_{t,1}-\pi_{0,1}|$ be the number of horizontal steps. Each horizontal step is followed by a stretch of vertical steps, and for $1\leq i\leq n$, let $\ell_i\in\Z$ denote the vertical displacement after the $i$-th horizontal step.
This gives a bijection between $\cI_t$ and $\bigcup_{n=1}^t \cL_{n,t}$, where
\begin{equation}
\label{stretches}
 \cL_{n,t} := \bigg\{\underline \ell=(\ell_1,\dots,\ell_n)\in\Z^n \colon \sum_{k=1}^j \ell_k\geq 0\,\, \forall\, j=1,\dots, n,\, \sum_{k=1}^n \ell_k= 0,\, n+\sum_{j=1}^n |\ell_j|=t
 \bigg\}.
\end{equation}
At this stage we note that
\begin{equation}\label{excursion2t}
\frac{1}{2^t} \big | \cI_{t}\big |=\sum_{\pi\in \cI_t} \frac{1}{2^{t-n(\pi)}}\frac{1}{3^{n(\pi)}}\, \Big(\frac{3}{2}\Big)^{n(\pi)}= \sum_{n=1}^t \sum_{\underline \ell\in \cL_{n,t}} \frac{1}{3^{n}}\frac{1}{2^{\sum_{j=1}^n|\ell_j|}} e^{n\, \log(\frac{3}{2})}.
\end{equation}
By identifying $\underline \ell=(\ell_1,\dots,\ell_n)$ in \eqref{excursion2t} with the increments of $V$, we get \eqref{defK}.
\end{proof}

\begin{figure}
\includegraphics[scale=2]{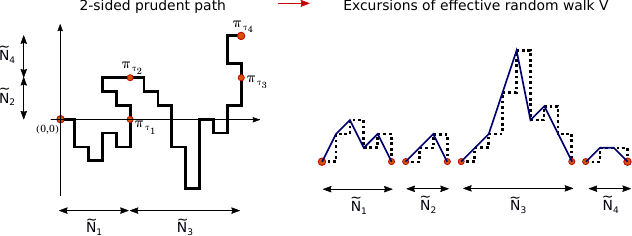}
\caption{We decompose a path $\pi\in \Omega_L^+$ into a sequence of horizontal and vertical excursions $\big((\pi_{\tau_{k-1}},\dots,\pi_{\tau_{k}})\big)_{k=1}^{4}$, each associated with an effective one dimensional random walk excursion.}
\label{fig1}
\end{figure}

\subsection{Representation of the law of a uniform 2-sided prudent walk}\label{secEL}

\begin{lemma}\label{lemma:Klamba}
Let $K$ be as in \eqref{defK}, then there exists $\lambda^*>0$ such that $\widehat K(\lambda^*):=\sum_{t=1}^\infty K(t) e^{-\lambda^* t} =1$.
\end{lemma}
\begin{remark}\label{remKK}
{\rm We will denote by $K^{*}$ the probability measure on $\N$ defined by
\begin{equation}
K^{*}(t)=K(t) e^{-\lambda^* t}, \qquad t\in\N.
\end{equation}
The proof of Lemma \ref{lemma:Klamba} below shows that there exists $\hat\lambda<\lambda^*$ such that $1<\widehat K(\hat \lambda)<\infty$. Therefore $K^*$ has exponential tail, i.e., there exist $c_1,c_2>0$ such that $K^{*}(n)\leq c_1 e^{-c_2 n}$ for every $n\in \N$.}
\end{remark}

\smallskip

The proof of Lemma \ref{lemma:Klamba} will be given at the end of the present section. We first explain how the law $K^*$ can be used to express the law $\bP^*$ of the excursions of the uniform two-sided prudent walk. Continuing Section \ref{sec:ERW}, let $\cV_\infty$ be the set of all
non-negative excursions of the effective walk, i.e.,
\be{Vinf}
\cV_\infty:=\bigcup_{N\geq 1}  \Big\{(V_i)_{i=0}^N\colon\, V_0=0, V_i\geq 0\ \forall i\leq N,\, V_N=0\Big\}.
\ee
By \eqref{defK} and Lemma \ref{lemma:Klamba}, we obtain the following probability law $\bP^*$ on $\cV_\infty$, with Radon-Nikodym derivative
\be{defPstar}
\frac{\dd \bP^*}{\dd \bP}\big((V_i)_{i=0}^N\big)=e^{\log(\frac{3}{2}) N-\lambda^{*} (N+\sum_{i=1}^N |U_i|)}.
\ee
We will show that $\bP^*$ is in fact the law of a uniform 2-sided prudent walk excursion. To that end, consider a sequence $(t_i,n_i)_{i=1}^r\in \N^r\times\N^r$ satisfying $t_1+\dots+t_r=L$ and $n_i\leq t_i$ for every $i\leq r$. Let $\Omega_L^+ \big( (t_i,n_i)_{i=1}^r \big)$ denote the
set of 2-sided prudent path consisting of $r$ excursions, where the $i$-th excursion has total length $t_i$, with $n_i$ horizontal (resp. vertical) steps if it is a horizontal (resp.\ vertical) excursion. By the reasoning leading to \eqref{two-sided1}, with $\alpha^*:=\log(3/2)-\lambda^*$, we obtain
\begin{equation}
\label{probeve1}
%\begin{split}
\frac{1}{2^L}|\Omega_L^+\big((t_i,n_i)_{i=1}^r\big)| e^{-\lambda^*L}=\prod_{i=1}^r \bE\bigg[e^{\alpha^* n_i-\lambda^*(t_i-n_i)} \,  \ind_{\big\{V_j \geq 0 \, \forall
j \leq n_i,\, V_{n_i}=0,\, n_i+\sum_{j=1}^{n_i} |U_j|=t_i \big\}}\bigg].
%\end{split}
\end{equation}
If $(\tilde T_i,\tilde N_i)_{i\in\N}$ denotes an i.i.d.\ sequence such that $\tilde N_1=N$ and $\tilde T_1=N+\sum_{i=1}^N |U_i|$ for a random walk excursion $(V_i)_{i=0}^N$ following the law  $\bP^*$ in \eqref{defPstar}, and
\be{gammaLV}
\tilde \gamma_L :=\min\{i\geq 1 \, :\, \tilde T_1+\dots +\tilde T_{i} \geq L\},
\ee
then by \eqref{two-sided1} and \eqref{probeve1}, for any set of paths $A$ which is a union of some $\Omega_L^+\big((t_i,n_i)_{i=1}^r\big)$, we have
\be{P+unif}
\bP_{\text{unif},L}^+ (A)=\frac{|\Omega_L^+ (A)| }{|\Omega_L^+|}
= \frac{\bE^*\bigg[\ind_A \ind_{\{\tilde T_1+\dots +\tilde T_{\tilde \gamma_L}=L\}}\bigg]}
{\bP^*\bigg[ \tilde T_1+\dots +\tilde T_{\tilde \gamma_L}=L\bigg]}, 
%=:\bP_L^*(A),
\ee
where we also used $\bP^*$ to denote the joint law of the i.i.d.\ sequence of effective random walk excursions that give rise to  $(\tilde T_i,\tilde N_i)_{i\in\N}$. This representation will be the basis of our analysis.
\medskip

%We left the proof of Lemma \ref{lemma:Klamba}.

\bpr[Proof of Lemma \ref{lemma:Klamba}]

The existence of
$\lambda^{*}$ is guaranteed if $\lambda^{**}:=\inf\{\lambda>0\colon\, \widehat K(\lambda) <\infty\}$ satisfies $\widehat K(\lambda^{**})>1$.
To show this, let $\tau$ be the first time the walk $V$ returns to or crosses the origin, i.e.,
\be{deftau}
\tau =\begin{dcases*}
	1 & if $V_1=0$,\\
  \min\{i\geq 2\colon\, V_{i-1} V_i\leq 0\} & otherwise.
  \end{dcases*}
\ee
Let $\alpha:=\log(3/2)-\lambda$. By \eqref{defK} and decomposing $V\in \cV_\infty$ into positive excursions, we can write
\begin{align}\label{dived}
\nonumber \widehat K(\lambda)&=\sum_{t\geq 1} \bE\Big[e^{(\log(\frac{3}{2})-\lambda ) \eta_t-\lambda (t-\eta_t)}\,   \ind_{\{V_i\geq 0\ \forall i\leq \eta_t, V_{\eta_t=0},\, \eta_t+\sum_{i=1}^{\eta_t} |U_i|=t\}}\Big]\\
\nonumber &= \sum_{t\geq 1} \sum_{N\leq t} \bE\Big[e^{ \alpha N-\lambda (t-N)} \,  \ind_{\{V_i\geq 0\ \forall i\leq N, \, V_N=0,\, N+\sum_{i=1}^{N} |U_i|=t\}}\Big]\\
\nonumber  &= \sum_{N=1}^{\infty} \bE\Big[e^{\alpha N} \, e^{-\lambda \sum_{i=1}^N |U_i|}  \ind_{\{V_i\geq 0\ \forall i\leq N, \, V_N=0\}}\Big]\\
\nonumber  &= \sum_{N=1}^\infty \sum_{r=1}^{\infty} \sum_{n_1+\dots+n_r=N} \prod_{i=1}^r \bE\Big[e^{\alpha  \tau} \, e^{-\lambda \sum_{i=1}^{\tau} |U_i|} \,  \ind_{\{V_1\geq 0,\, \tau=n_i,\, V_{n_i}=0\}}\Big]\\
\nonumber &= \sum_{r=1}^{\infty} \Big(\sum_{n=1}^\infty \bE\Big[e^{\alpha \tau -\lambda \sum_{i=1}^{\tau} |U_i|} \, \ind_{\{V_1\geq 0, \, \tau=n,\, V_{\tau}=0\}}\Big]\Big)^r\\
&=
\sum_{r=1}^{\infty} \left( \bE\Big[e^{\alpha \tau-\lambda \sum_{i=1}^{\tau} |U_i|}\,   \ind_{\{V_1\geq 0, V_{\tau}=0\}}\Big]\right)^r
=: \sum_{r=1}^\infty G(\lambda)^r.
\end{align}
Therefore $\lambda^{**}=\inf \{\lambda>0\colon\, G(\lambda)< 1\}$, and it suffices to show that $G(\lambda^{**})>1/2$. Note that
\begin{equation}\label{G1}
\bE\Big[e^{\alpha \tau-\lambda \sum_{i=1}^{\tau} |U_i|}\,   \ind_{\{V_1= 0\}}\Big]=\frac{e^\alpha}{3},
\end{equation}
and
\begin{equation}\label{G2}
\bE\Big[e^{\alpha \tau-\lambda \sum_{i=1}^{\tau} |U_i|}\,   \ind_{\{V_1>0, \tau=n\}}\Big] = \bE\Big[e^{\alpha \tau-\lambda \sum_{i=1}^{\tau} |U_i|}\,   \ind_{\{V_1>0, \tau=n, V_\tau=0\}}\Big] \frac{1}{1-e^{-\lambda}/2} ,
\end{equation}
because given $(V_i)_{i=0}^{n-1}$ with $V_1>0$, the events $\{\tau=n, V_n=0\}$ and $\{\tau=n\}$ differ only in that the first event requires $U_n=-V_{n-1}$, while the second event requires $U_n\leq -V_{n-1}$, and the probability ratio of the two events is precisely
$\sum_{k=0}^\infty \frac{e^{-k\alpha}}{2^k} = \frac{1}{1-e^{-\lambda}/2}$ by \eqref{lawP}. Summing over $n$ in \eqref{G2}, using the symmetry of $V$ and \eqref{G1} then gives
\begin{align}\label{Glambda}
G(\lambda)&= \frac{e^\alpha}{3}\Big(\frac{1}{2}+\frac{e^{-\lambda}}{4}\Big)+\frac{1}{2} \Big(1-\frac{e^{-\lambda}}{2}\Big)
\bE\Big[ e^{\alpha \tau-\lambda \sum_{i=1}^{\tau} |U_i|}\Big].
\end{align}

Now let $\hat \lambda$ be the unique solution of
$$
\log \bE[e^{-\lambda |U_1|}]=-\alpha=\lambda-\log(3/2), \qquad \lambda \in [0, \infty).
$$
Then $(M_n^{\hat\lambda})_{n\geq 0}:=(e^{\alpha n-\hat \lambda \sum_{i=1}^n |U_i|})_{n\geq 0}$ is a positive martingale. We will show that $\bE[M_\tau^{\hat\lambda}]=1$, which then gives $G(\hat \lambda)= \frac{1}{2}+\frac{e^{-2\hat\lambda}}{8}\in (1/2,1)$. By definition, we have $\hat\lambda> \lambda^{**}$. Since $\lambda\mapsto G(\lambda)$ is strictly decreasing, we conclude that $G(\lambda^{**})> G(\hat\lambda)>1/2$.
\smallskip

It remains to prove that $\bE[M_\tau^{\hat\lambda}]=1$. Note that  $\tau$ is an almost surely finite stopping time, so that $M^{\hat\lambda}_{n\wedge \tau}$ converges almost surely to $M_\tau^{\hat\lambda}$. Fatou's lemma implies  $\bE[M_{\tau}^{\hat\lambda}]\leq 1$. On the other hand,
\begin{align}\label{compuG}
\bE[M_{\tau}^{\hat\lambda}] = \lim_{n\to \infty}  \bE[M_{\tau}^{\hat\lambda} \, \ind_{\{\tau \leq n\}}]
= \lim_{n\to \infty} \big(1-\bE[M_{n\wedge \tau}^{\hat\lambda} \, \ind_{\{\tau>n\}}]\big).
\end{align}
It remains to prove that $\lim_{n\to \infty} \bE[M_{n}^{\hat\lambda} \ind_{\{\tau>n\}}]=0$. Let $(\widetilde U_i)_{i\geq 1}$ be i.i.d.\  with law $\widetilde \bP$ such that
$$\widetilde \bP(\widetilde U_1=x)= \frac{1}{\bE[e^{-\hat \lambda |U_1|}]} \, e^{-\hat \lambda |x|} \, \bP(U_1=x), \quad x\in \Z.$$
We observe that
\begin{align}\label{compudd}
\bE[M_{n}^{\hat\lambda} \, \ind_{\{\tau>n\}}]=e^{\alpha n +\log \bE[ e^{-\hat \lambda |U_1|}] n} \, \widetilde \bP(\tau>n)
=\widetilde \bP(\tau>n).
\end{align}
Under $\widetilde \bP$, the random walk increments $(\widetilde U_i)_{i\geq 1}$ are symmetric and integrable. Thus,  $\tau$ is finite $\widetilde \bP$-a.s.\ and the right hand side in \eqref{compudd} converges to $0$ as
$n$ tends to $\infty$.  We conclude that  $\bE[M_{\tau}^{\hat\lambda}]=1$.
\epr

\subsection{Scaling limit of the uniform 2-sided prudent walk}

In this section we prove Theorems \ref{scalingtwosided} and \ref{scalingtwosided2}.

%\smallskip

\bpr[Proof of Theorems \ref{scalingtwosided} and \ref{scalingtwosided2}]
Let $\bP^*$ be the law of the i.i.d.\ sequence of effective random walk excursions as in \eqref{defPstar}, and let $(\tilde T_i,\tilde N_i)_{i\in\N}$ and $\tilde \gamma_L$ be as introduced after \eqref{probeve1}. Then by the law of large numbers, as $L\to\infty$, almost surely we have $ \frac{\tilde \gamma_L}{L} \to \frac{1}{\bE^*[\widetilde T_1]}>0$, since $\tilde T_1$ has exponential tail by Remark \ref{remKK}. Let $\tilde\tau_k =\sum_{i=1}^k \tilde T_i$, which defines a renewal process. For any $t_0< 1/\bE^*[\widetilde T_1]$, note that by the renewal theorem, cf. \cite[Appendix A]{GB07}, the law of $(\tilde T_i, \tilde N_i)_{1\leq i\leq t_0L}$ conditioned on $L\in \tilde \tau$ is equivalent to its law under $\bP^*$ without conditioning, in fact their total variation distance tends to $0$ as $L$ tends to infinity since $L-\sum_{i=1}^{t_0L} \tilde T_i\to\infty$ in probability. Therefore to identify the scaling limit of $(\pi_i)_{i=1}^{t_0L}$ under $\bP_{\text{unif},L}^+$, by \eqref{P+unif}, it suffices to consider $\bP^*$ in place of $\bP_{\text{unif},L}^+$.
\medskip

Recall that the 2-sided uniform prudent walk $\pi$ is constructed by concatenating alternatingly eastward horizontal excursions and northward vertical excursions, where modulo rotation, the excursions have a one-to-one correspondence with the effective random walk excursions. Therefore if we let $X_n:=(X_{n,1}, X_{n, 2})$ be a random walk on $\Z^2$ with
\begin{equation}\label{Xn12}
X_{n, 1} = \sum_{i=1}^n (\tilde N_{2i-1}-c(\tilde T_{2i-1}+\tilde T_{2i})), \  \ X_{n,2} = \sum_{i=1}^n (\tilde N_{2i}-c(\tilde T_{2i-1}+\tilde T_{2i})), \quad \mbox{where} \ c= \frac{\bE^*[\widetilde N_1]}{2\bE^*[\widetilde T_1]},
\end{equation}
then $X_n = \pi_{\phi(n)}-c\phi(n)\vec e_1$, with $\phi(n) =\sum_{i=1}^{2n} \tilde T_{i}$ playing the role of time change. By the strong law of large numbers, $\bP^*$-a.s., we have
\begin{equation}\label{phiconv}
\Big(\frac{1}{L} X_{tL}\Big)_{t\geq 0} \to 0 \qquad \mbox{and} \qquad \Big(\frac{\phi(tL)}{L}\Big)_{t\geq 0} \to \big(2t\bE^*[\tilde T_1]\big)_{t\geq 0}.
\end{equation}
It is then easily seen that, with $I:=\{\frac{1}{L}\sum_{i=1}^{2k} \widetilde T_i : 1\leq k\leq t_0L/2\}$, the rescaled path $\tilde \pi^L$ satisfies
\begin{equation}\label{supconv}
\sup_{t\in I}\big|\tilde \pi^L_t-c t \, \vec{e}_1 \big| = \sup_{t\in I} \Big|\frac{1}{L}X_{\phi^{-1}(tL)} \Big|
 \to 0 \qquad \bP^*\mbox{-}a.s.\ \mbox{ as } L\to\infty.
\end{equation}
In fact \eqref{supconv} still holds if the supremum is taken over all $0\leq t\leq \frac{1}{L} \sum_{i=1}^{t_0L} \widetilde T_i$, since for the $i$-th excursion, the prudent path deviates from the end points of the excursion by at most $\widetilde T_i$, which has exponential tail by Remark \ref{remKK}. It is then easily seen that
\begin{equation}\label{Lmax}
\frac{1}{\sqrt{L}} \max_{1\leq i\leq L} \widetilde T_i \to 0 \qquad \bP^*\mbox{-}a.s.\ \mbox{ as } L\to\infty.
\end{equation}
Therefore \eqref{supconv} holds with $\sup$ taken over $t\in [0,\tilde t_0]$, with $\tilde t_0:=\lim_{L\to\infty}\frac{1}{L} \sum_{i=1}^{t_0L} \widetilde T_i = t_0\bE^*[\tilde T_1]<1$, and $(\tilde \pi^L_t)_{t\in [0,\tilde t_0]}$ converges in probability to $(ct\vec e_1)_{t\in [0, \tilde t_0]}$ under $\bP^*$ as well as $\bP_{{\rm unif},L}^+$. We can now deduce \eqref{scalingtwosided} by letting $\tilde t_0\uparrow 1$, using that modulo time reversal, translation and rotation,  $(\pi_i)_{i=\gamma_L-\epsilon L}^{\gamma_L}$ has the same law as $(\pi_i)_{i=1}^{\epsilon L}$ under $\bP_{\text{unif},L}^+$, and hence is negligible in the scaling limit as $\epsilon\downarrow 0$.
\medskip

The proof of Theorem \ref{scalingtwosided2} is similar. By \eqref{Lmax}, it suffices to consider $\pi_t-ct\vec e_1$ along the sequence of times $(\phi_n)_{n\in\N}$, which is a time change of $(X_n)_{n\in\N}$. It is clear that $(X_{tL}/\sqrt{L})_{t\geq 0}$ converges to a Brownian  motion $(\tilde B_t)_{t\geq 0}$ with covariance matrix $\bE[\tilde B_{1, i} \tilde B_{1, j}] = \bE[X_{1, i} X_{1, j}]$. Undo the time change $\phi$, which becomes asymptotically deterministic by \eqref{phiconv}, we find that under $\bP^*$, hence also $\bP_{\text{unif},L}^+$,
$$
\sqrt{L}(\tilde \pi^L_t - ct\vec e_1)_{t\in [0, \tilde t_0]} \Rightarrow (B_t)_{t\in [0, \tilde t_0]},
$$
where $B$ is a Brownian motion with covariance matrix
\begin{equation}\label{covar}
\bE[B_{1, i} B_{1, j} ] = \frac{\bE\big[\big(2\tilde N_i\bE^*[\tilde T_1] -\bE^*[\tilde N_1](\tilde T_1 +\tilde T_2)\big)\big(2\tilde N_j\bE^*[\tilde T_1] -\bE^*[\tilde N_1](\tilde T_1 +\tilde T_2)\big)\big]}{8\bE^*[\tilde T_1]^3}, \  i, j=1, 2. \!\!\!
\end{equation}
Letting $\tilde t_0\uparrow 1$ and applying the same reasoning as before then gives \eqref{concentr2sidedclt}.
\epr

\section{Uniform prudent walk}\label{pw}

By symmetry, we may assume without loss of generality that the prudent walk starts with an east step, and the first vertical step is a north step. We will assume this from now on.

\subsection{Decomposition of a prudent path into excursions in its range}\label{decppe}
We now decompose each prudent path $\pi\in \Omega_L$ into a sequence of excursions within its range (see Figure \ref{fig2}). We use the same decomposition as in \cite[Section 2]{BFV10}, which is slightly different from our decomposition for the 2-sided prudent path.
\smallskip

For every $t\leq L$, let $\cA_t$ (resp. $\cB_t$) denote the projection of the range of $\pi$ onto the $x$-axis (resp. $y$-axis), i.e.,
\begin{equation}
\cA_t =\big\{ \pi_{i, 1}\in \Z: 0\leq i\leq t\} \qquad \text{and}\qquad
\cB_t =\big\{ \pi_{i, 2}\in \Z: 0\leq i\leq t\}.
\end{equation}
Let $\cW_t=|\cA_t|$ and $\cH_t=|\cB_t|$ denote respectively the width and height of the range $\pi_{[0,t]}$.  Define $\cH_0=\cW_0=1$, and set $\rho_0=\nu_0=0$. For $k\geq 0$, define
\begin{align}
&\rho_{k+1}=\min\{t>\upsilon_{k} \, :\, \cH_t>\cH_{t-1}\}-1,
&\upsilon_{k+1}=\min\{t>\rho_{k+1}\,:\, \cW_t>\cW_{t-1}\}-1.
\end{align}
We say that on each interval $[\rho_k,\upsilon_k]$ (resp.\  $[\upsilon_k,\rho_{k+1}]$) $\pi$ performs a vertical (resp.\ horizontal) excursion in its range, and the path is monotone in the vertical (resp.\ horizontal) direction. Note that each excursion ends by exiting one of two sides of the smallest rectangle containing the range of $\pi$ up to that time, and the excursion ends at a corner of this rectangle.

\smallskip
Let $\gamma_L(\pi)$ be the number of complete excursions contained in $\pi$, where the last excursion is considered complete if adding an extra horizontal or vertical step can make it complete. Let $T_i$ denote the length of the $i$-th excursion, $N_i$ its horizontal (resp.\ vertical) extension if it is a horizontal (resp.\ vertical) excursion, and let $\cE_i=1$ if the excursion crosses the range and let $\cE_i=0$ otherwise. More precisely, a horizontal excursion on the interval $[\nu_k, \rho_{k+1}]$ crosses the range if $|\pi_{\rho_{k+1}, 2}-\pi_{\nu_k, 2}|= \cH_{\rho_{k+1}}$. We can thus associate with every $\pi\in \Omega_L$ the sequence $(T_i,N_i,\cE)_{i=1}^{\gamma_L(\pi)}$.
Note that the $i$-th excursion is a horizontal excursion if $i$ is odd, and vertical excursion if $i$ is even.
For $i\in\N$, let $R_{i-1}$ denote the width (resp.\ height) of the range of $\pi$ before the start of the $i$-th excursion if it is a vertical (resp.\ horizontal) excursion. It can be seen that $R_{i}=R_{i-2}+N_i$ for $i\geq 1$, with $R_{-1}=R_0=0$.

%For $k\in\N$ with $2k+1\leq \gamma_L(\pi)$, denote $R_{2k+1}=\sum_{i=0}^{k} N_{2i+1}$, and for $k\in\N$ with $2k\leq \gamma_L(\pi)$, denote $R_{2k}=\sum_{i=1}^k N_{2i}$. Then for any $k\in \N_0$, we have $\cW_{\rho_k}=R_{2k}+1$ and $\cH_{\upsilon_k}=R_{2k+1}+1$. {\bf see figure...}

\begin{figure}
\includegraphics[scale=2.5]{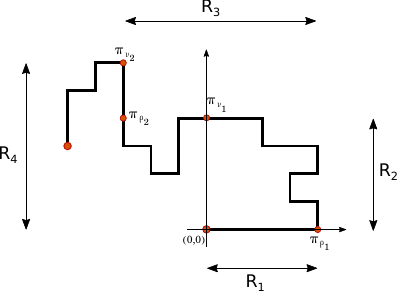}
\caption{We decompose a path $\pi\in \Omega_L$ into a sequence of excursions. 
The $i$-th excursion is a horizontal excursion if $i$ is odd, and vertical excursion if $i$ is even.
The $1$-st excursion corresponds to the sub-path $\pi_{[0,\rho_1]}$, the $2$-nd to the sub-path $\pi_{[\rho_1,\upsilon1]}$ and so on. At the end of the $i$-th excursion, if $i$ is odd (resp. if $i$ is even) we set $R_i$ to be the width (resp. the height) of the range.
The last excursion is incomplete. 
}
\label{fig2}
\end{figure}

\subsection{Effective random walk excursion in a slab}\label{effslab}

The one-to-one correspondence in Section \ref{sec:ERW} between the excursion paths (which are partially directed) and the effective random walk paths can be extended to the current setting, except that now the effective random walk lies in a slab corresponding to the range of the path at the start of the excursion, and the excursion may end on either side of the slab. As a consequence, we define a measure $L_R$ on $\N\times \N\times \{0,1\}$ by
\begin{equation}\label{LR}
\begin{aligned}
 L_{R}(t,n,0)&=\bE\Big[e^{\log(\frac{3}{2}) n-\lambda^*t} \, \ind_{\{V_i\in \{0,\dots,R\}\  \forall i\leq n,\,  V_n=0,\,  \sum_{i=1}^n |U_i|=t-n\}}\Big],\\
L_{R}(t,n,1)&=\bE\Big[e^{\log(\frac{3}{2}) n-\lambda^*t} \, \ind_{\{V_i\in \{0,\dots,R\}\  \forall i\leq n,\,  V_n=R,\,  \sum_{i=1}^n |U_i|=t-n\}}\Big].
\end{aligned}
\end{equation}
When $R=0$, define $L_0(t,n, 1)$ as above and define $L_0(t, n, 0)=0$. Let $\widehat L_R$ be a variant of $L_R$ that accounts for an incomplete excursion (cf. Figure \ref{fig2}), i.e.,
\begin{align}\label{hatLR}
\widehat L_R(t,n)&= \bE\Big[e^{\log(\frac{3}{2}) n-\lambda^*t} \, \ind_{\{V_i\in \{0,\dots,R\}\  \forall i\leq n,\, 0< V_n<R, \, \sum_{i=1}^n |U_i|=t-n\}}\Big],
\end{align}
where $\lambda^*$ is as in Lemma \ref{lemma:Klamba}. We also set $\widehat{L}_R(t)=\sum_{n\geq 1}\widehat L_R(t,n)$ and $\widehat{L}_R(0)=1$.
\smallskip

Let $\alpha^*=\log(\frac{3}{2})-\lambda^*$, and let $(t_i, n_i, \gep_i)\in \N^2 \times \{0,1\}$, $1\leq i\leq r$, be such that $t_1+\dots+t_r\leq L$ and $n_i\leq t_i$. Let $\Omega_L \big( (t_i,n_i,\gep_i)_{i=1}^r \big)$ be the set of prudent paths containing $r$ complete excursions, with
$(T_i,N_i,\cE_i)_{i=1}^r=(t_i,n_i,\gep_i)_{i=1}^r$, and recall $(R_{i-1})_{i\in\N}$ from the end of Section \ref{decppe}. Reasoning as for \eqref{probeve1}, we then have
\begin{align}
\frac{1}{2^L} |\Omega_L \big( (t_i,n_i,\gep_i)_{i=1}^r \big)| e^{-\lambda^* L}=&
\prod_{i=1}^r \bE\bigg[e^{\alpha^* n_i-\lambda^*(t_i-n_i)} \,  \ind_{\big\{V_i\in [0,R_{i-1}] \, \forall
i \leq n_i,\, V_{n_i}=\gep_i R_{i-1},\, n_i+\sum_{j=1}^{n_i} |U_j|=t_i \big\}}\bigg]  \nonumber \\
& \quad \times \, \widehat{L}_{R_r}\big(L-(t_1+\dots+t_r)\big)  \label{tug}\\
=&\bigg[\prod_{i=1}^r  L_{R_{i-1}}(t_i,n_i,\gep_i)\bigg]\, \widehat{L}_{R_r}\big(L-(t_1+\dots+t_r)\big), \nonumber
\end{align}
where $\widehat L_{R_r}(L-(t_1+\dots+t_r))$ accounts for the last incomplete excursion in $\pi$.

\subsection{Representation of the law of a uniform prudent walk}
\label{sec:sampling}
We now show how to represent the law of the uniform prudent walk in terms of the excursions of the effective random walk $V$.
\smallskip

For $R\in \N$, let $\cV_R$ be the set of effective random walk paths in a slab of width $R$ and ending at either $0$ or $R$. Namely,
\be{sample}
\cV_R:=\bigcup_{N\geq 1}\Big[ \cV^{\, 1}_{N,R} \cup  \cV^{\, 0}_{N,R}\Big],
\ee
where for $a=0, 1$,
\begin{align}\label{VR}
\cV_{N,R}^{\, a}:= \Big\{(V_i)_{i=0}^N\colon\, V_0=0, V_i\in \{0,\dots,R\}\ \forall i\in \{0,\dots,N\}, V_N=aR\Big\}.
\end{align}
Recall the effective random walk excursion measure $\bP^*$ from \eqref{defPstar}. We will define a probability law $\bP_R^{*}$ on $\cV_R$ by sampling a path under $\bP^*$ and truncating it if it passes above $R+1$. More precisely, define the truncation $T_R:\cV_\infty\mapsto \cV_{R}$ as follows. Given $V:=(V_i) _{i=0}^N\in \cV_\infty$, let $T_R V := V$ if $V_i\leq R$ for every $i\leq N$. Otherwise, let $\tau_R:=\inf\{i\geq 1\colon V_i\geq R+1\}$ and set
\be{deftauR}
(T_R V)_i=V_i \quad \text{for $i\leq \tau_R-1$ and}\  (T_RV)_{\tau_R}=R.
\ee
Then define $\bP_R^*$ as the image measure of $\bP^*$ under $T_R$. For each trajectory $V\in \cV_R$, we associate $(T,N,\cE)$ such that $N$ is the number of increments $(U_i)_{i=1}^N$ of $V$, $T=N+\sum_{i=1}^N |U_i|$, and $\cE=1$ if $V_N=R$ and $\cE=0$ if $V_N=0$ (if $R=0$, set $\cE=1$). Let $L^*_{R}$ denote the law of $(T,N,\cE)$ when $V$ is sampled from $\bP_R^*$, and we observe that $L_R^*$ and $L_R$ (cf. \eqref{LR}) coincide when $\gep=0$, i.e.,
\be{L0=L*0}
L_R(t,n,0)=L_R^{*}(t,n,0),\quad  (t,n)\in \N\times \N.
\ee

\smallskip
Let $(\widetilde V^{(i)})_{i\geq 1}$ be an i.i.d.\ sequence of effective walk excursions with law $\bP^{*}$, and for each $i\in\N$, let
$(\widetilde T_i, \widetilde N_i)$ denote the total length and the number of increments of $\widetilde V^{(i)}$. We now construct a sequence $(T_i,N_i,\cE_i)_{i\geq 1}$ from $(\widetilde V^{(i)})_{i\geq 1}$ inductively, using the truncation map $T_R$. First set $R_{-1}=R_0:=0$. For each $i\geq  1$, set
\be{defTNV}
V^{(i)}=T_{R_{i-1}}\widetilde V^{(i)},\qquad (N_i,T_i,\cE_i)=(N, T, \cE)(V^{(i)}),\quad \text{and} \quad R_i=R_{i-2}+N_i.
\ee
where $(N,T,\cE)(V^{(i)})$ is the triple $(N, T, \cE)$ associated with $V^{(i)}\in \cV_{R_{i-1}}$. For every $i\geq 1$, we have $N_i\leq \widetilde N_i$ and $T_i\leq \widetilde T_i$, and conditioned on $(T_j,N_j,\cE_j)_{j=1}^{i-1}$, the law of $(N_i,T_i,\cE_i)$ is $\bP_{R_{i-1}}^{*}$. Note that the excursion decomposition of a prudent path in Section \ref{decppe} gives exactly a sequence of excursions of the form $(T_R\widetilde V^{(i)})_{i\geq 1}$.
\medskip

For a set of prudent paths $A\subset \Omega_L$ depending only on $(T_i,  N_i, \cE_i)_{i=1}^{\gamma_L}(\pi)$, where
\begin{equation}
\gamma_L=\min\{i\geq 1\colon\, T_1+\dots+T_i> L\}-1,
\end{equation}
let $(t_i, n_i, \gep_i)_{i=1}^r\sim A$ denote compatibility with $A$. By \eqref{tug}, we then have
\begin{align}\label{probeve}
& \frac{1}{2^L} |A| e^{-\lambda^* L} = \sum_{(t_i, n_i, \gep_i)_{i=1}^r \sim A} \bigg[\prod_{i=1}^r  L_{R_{i-1}}(t_i,n_i,\gep_i)\bigg]\, \widehat{L}_{R_r}\big(L-(t_1+\dots+t_r)\big) \nonumber\\
= \ & \bE^*\Bigg[\ind_{\{( T_i,    N_i,   \cE_i)_{i=1}^{\gamma_L}\sim A\}} \prod_{i=1}^{\gamma_L}\   \frac{L_{R_{i-1}}( T_i,  N_i,  \cE_i)}{L^*_{R_{i-1}}(  T_i,  N_i, \cE_i)} \, \cdot \,
 \frac{\widehat{L}_{R_{\gamma_L}}(L-( T_1+\dots+  T_{\gamma_L}))}{\bP_{R_{\gamma_L}}^*\big( T> L-( T_1+\dots+T_{\gamma_L}\big)}\Bigg],
\end{align}
where $\bE^*$ is expectation over the i.i.d.\ excursions $(\widetilde V^{(i)})_{i\geq 1}$, and hence $( T_i,  N_i, \cE_i)_{i\geq 1}$.

\medskip

We conclude this section with two technical lemmas needed to control the ratios inside the expectation in \eqref{probeve}.
For ease of notation, let us denote
\begin{equation}
L_R(t,\epsilon):=\sum_{n\geq 1} L_R(t,n,\epsilon) \qquad \mbox{and}  \qquad L_R(t):=L_R(t,0)+L_R(t,1).
\end{equation}

\begin{lemma}\label{Cbonud}
There exists $C>0$ such that
\be{rnd}
\frac{L_R(t)}{L_R^*(t)}\leq C\,  t \, \ind_{\{t\geq R\}}\, +\, \ind_{\{t<R\}} \qquad \mbox{for all } t\in \N.
\ee
\end{lemma}
\bpr
First, observe that for $t<R$, a path of length $t$ cannot reach level $R$. Therefore, $L_{R}(t,n,1)=L^*_R(t,n,1)=0$ and $L_{R}(t,n,1)=L^*_R(t,n,1)$. It only remains to consider $t\geq R$, and it suffices to show that $L_R(t,1)\leq C t  L_{R}(t,0)=C tL_R^*(t)$. For simplicity we only consider the case $R\in 2\N$, but the case $R\in 2\N+1$ can be treated in a similar manner. Let
\begin{align}\label{ABR}
\cB_{n,t}^{R}:=& \Big\{(V_i)_{i=0}^n\colon\, V_0=0, V_i\in \{0,\dots,R\}\ \forall i\in \{0,\dots,n\}, V_n=R, \sum_{i=1}^n|U_i|=t-n\Big\},\\\label{HR1}
\cA_{n,t}^{R}:=& \Big\{(V_i)_{i=0}^n\colon\, V_0=0, V_i\in \{0,\dots,R\}\ \forall i\in \{0,\dots,n\}, V_n=0,  \sum_{i=1}^n|U_i|=t-n\Big\}.
\end{align}
We define a map $G_{n,t}^R: \cB_{n,t}^R\mapsto \cA_{n,t}^R\cup \cA_{n+2,t}^R$ as follows. For $V\in \cB_{n,t}^R$, let $\tau_{R/2}:=\min\{i\geq 1\colon\, V_i\geq R/2\}$. We distinguish between two cases (see Figure \ref{fig3}):
\begin{enumerate}
\item If $V_{\tau_{R/2}}=R/2$, then define $G_{n,t}^R(V)$ by simply reflecting $V$ across $R/2$ from $\tau_{R/2}$ onward, i.e.,
$G_{n,t}^R (V)_i=V_i$ for $i\leq \tau_{R/2}$ and  $G_{n,t}^R(V)_i=R-V_i$ for $i\in \{\tau_{R/2},\dots,n\}$. Then,  $G_{n,t}^R (V)\in
\cA_{n,t}^{R}$.\\
\item If $V_{\tau_{R/2}}=R/2+y$ with $y\in \{1,\dots,\frac{R}{2}\}$, then let $G_{n,t}^R(V)_i=V_i$ for $i\leq \tau_{R/2}-1$,
$G_{n,t}^R (V)_{\tau_{R/2}}=\frac{R}{2}-1$, $G_{n,t}^R(V)_i=R-V_{i-1}$ for $i\in \{\tau_{R/2}+1,\dots,n+1\}$ and $G_{n,t}^R(V)_{n+2}:=0$. Then,  $G_{n,t}^R(V)\in \cA_{n+2,t}^R$.
\end{enumerate}
Note that under $G_{n, t}^R$, every $V\in G_{n,t}^R(\cB_{n,t}^R)\cap \cA_{n,t}^R$ has a unique pre-image in $\cB_{n,t}^R$, and every $V\in G_{n,t}^R(\cB_{n,t}^R)\cap \cA_{n+2,t}^R$ has at most $n\leq t$ pre-images in $\cB_{n,t}^R$, one for each time that $V$ is at level $\frac{R}{2}-1$. Finally, we note that in the second case,  $G_{n,t}^R(V)$ has two fewer vertical steps and two more horizontal steps than $V$. This allows us to write
\begin{align}\label{grandb}
L_R(t,1)&=\sum_{n=1}^t \bE\Big[ e^{\log(\frac{3}{2}) n-\lambda^* t} \ \ind_{\cB_{n,t}^R}(V)\Big]\\
\nonumber &\leq \sum_{n=1}^t  \bE\Big[ e^{\log(\frac{3}{2}) n-\lambda^* t} \ \ind_{\cA_{n,t}^R}(V)\Big]+
t\,  \bE\Big[ e^{\log(\frac{3}{2}) (n+2)-\lambda^* t} \ \ind_{\cA_{n+2,t}^R}(V)\Big].
\end{align}
Observe that the r.h.s.\ in \eqref{grandb} is less than $\sum_{n=1}^t  L_R(n,t,0)+t L_R(n+2,t,0)$, which implies
\begin{align}\label{grandb2}
L_R(t,1)&\leq 2 t L_R(t,0).
\end{align}
This concludes the proof of the lemma.
\epr

\begin{figure}[t]
\begin{subfigure}{1\textwidth}
\centering
\includegraphics[scale=1.4]{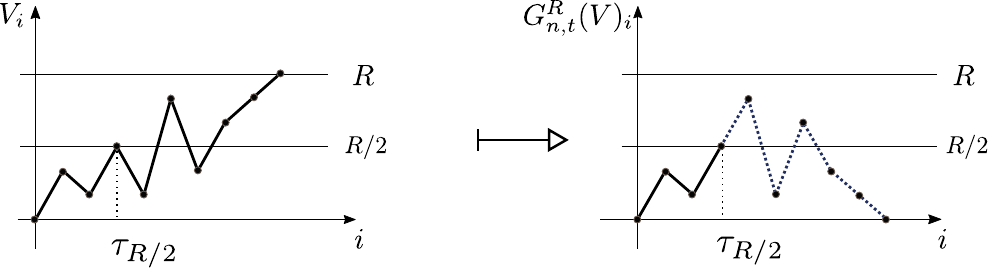}
\label{subfig31}
\caption{}
\end{subfigure}

\vspace{0.7cm}

\begin{subfigure}{1\textwidth}
\centering
\includegraphics[scale=1.4]{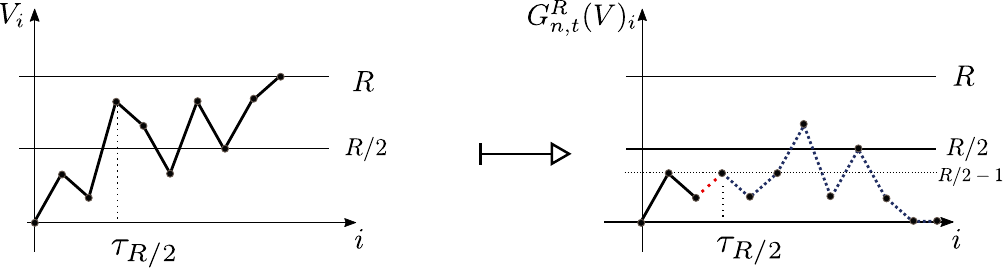}
\caption{}
\label{subfig32}
\end{subfigure}
\caption{The transformation $G_{n,t}^R(V)$. We let $\tau_{R/2}:=\min\{i\ge 1\colon V_i \ge R/2\}$. 
In $(\textrm a)$ we draw the case in which $V_{\tau_{R/2}}=R/2$. In this case we define
 $G_{n,t}^R(V)$ by simply reflecting $V$ across $R/2$ from $\tau_{\tau_{R/2}}$ onward (in blue, dotted). 
 In $(\textrm b)$ we draw the case in which $V_{\tau_{R/2}}> R/2$. In this second case we let $G_{n,t}^R (V)_{\tau_{R/2}}=\frac{R}{2}-1$ (in red, dotted) and we concatenate the reflection of $V$ across $R/2$ from $\tau_{R/2}$ onward. We add a final point $G_{n,t}^R(V)_{n+2}:=0$ (in blue, dotted).
 }
 \label{fig3}
\end{figure}

\medskip

\begin{figure}
\centering
\includegraphics[scale=1.5]{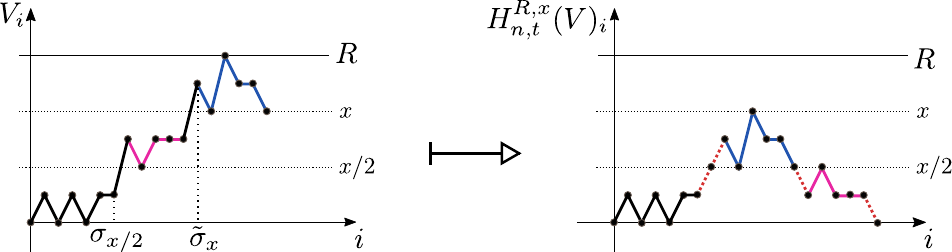}
\caption{The transformation $H_{n,t}^{R,x}(V)$: We fix $x\le R$ even and we consider a path $V$ ending at $V_n=x$. We let
$\sigma_{x/2} :=\max\{i\geq 0\colon\, V_i<\tfrac{x}{2}\}$ and $\tilde \sigma_{x} :=\min\{i\geq \sigma_{x/2}+1 \colon\, V_i\geq x\}.$
In the figure we draw the transformation $H_{n,t}^{R,x}(V)$ when $V_{\sigma_{x/2}+1}>x/2$ and $V_{\widetilde \sigma_{x}}>x$.
In this case we define $H_{n,t}^{R,x}(V)_{\sigma_{x/2}+1}:=x/2$. Then we take the piece of $V$ on the interval $[\widetilde \sigma_x, n]$
 lowered by $x/2$ and we insert it at time $\sigma_{x/2}+2$ (blue). Finally the piece of $V$ on the interval $[\sigma_{x/2}+1, \widetilde \sigma_x-1]$ is reflected across $x/2$ and reattached at the end (violet). We add a final point $H_{n,t}^{R,x}(V)_{n+2}:=0$.}
\label{fig4}
\end{figure}

To bound the last ratio in \eqref{probeve}, we will bound $\widehat{L}_R(t)/ \bP_R^*\big( T \geq t\big)$, which arises from the last incomplete excursion in the excursion decomposition. Recall that $\widehat{L}_R(0):=1=\bP_R^*(T> 0)$.

\begin{lemma}\label{Cbonudlast}
There exists $C>0$ such that
\be{rnd2}
\frac{\widehat L_R(t)}{\bP_R^*( T> t)}\leq CRt^2  \quad \mbox{ for all } R, t\in \N.
\ee
\end{lemma}
\bpr
Recall $\widehat L_R(t)$ from \eqref{hatLR}. It suffices to show that there exists $C>0$ such that
\be{supnu}
\widehat L_R(t)\leq C R\,  t^2 L_R(t+2,0),
\ee
since
$$
\bP_R^*\big( T > t\big)=\sum_{j> t}L^*_R(j)\geq \sum_{j> t} L_R^*(j,0)=\sum_{j> t} L_R(j,0)\geq L_R(t+2,0).
$$
For $x\in \{1,\dots,R-1\}$ and $n\leq t$, we consider the set of effective random walk trajectories
\begin{align}
\cD_{n,t}^{R,x}=& \Big\{(V_i)_{i=0}^n\colon\, V_0=0, V_i\in \{0,\dots,R\}\ \forall i\in \{0,\dots,n\}, V_n=x, \sum_{i=1}^n|U_i|=t-n\Big\}.
\end{align}
For simplicity, we assume that $x$ is even, but the case $x$ odd can be treated similarly. Let
\be{defsigma}
\sigma_{x/2} :=\max\{i\geq 0\colon\, V_i<\tfrac{x}{2}\}\quad \text{and}\quad
\tilde \sigma_{x} :=\min\{i\geq \sigma_{x/2}+1 \colon\, V_i\geq x\}.
\ee

We define a map $H_{n,t}^{R,x}:\cD_{n,t}^{R,x}\to \cA_{n+2,t+2}^{R}$ (cf. \eqref{HR1}) as follows. Let $V\in\cD_{n,t}^{R,x}$. We distinguish between four cases:
\begin{enumerate}
\item $V_{\sigma_{x/2}+1}>x/2$ and $V_{\widetilde \sigma_{x}}>x$,
\item $V_{\sigma_{x/2}+1}>x/2$ and $V_{\widetilde \sigma_{x}}=x$,
\item $V_{\sigma_{x/2}+1}=x/2$ and $V_{\widetilde \sigma_{x}}>x$,
\item $V_{\sigma_{x/2}+1}=x/2$ and $V_{\widetilde \sigma_{x}}=x$.
\end{enumerate}
We will treat case 1 only, where $H_{n,t}^{R,x}$ maps $V$ to a path in $\cA_{n+2,t+2}^{R}$ (see Figure \ref{fig4}). Cases 2--4 are similar and even simpler, and to ensure
that $H_{n,t}^{R,x}(V) \in \cA_{n+2,t+2}^{R}$, we can add extra horizontal steps if needed. Roughly speaking, under $H_{n,t}^{R,x}$, the piece of $V$ on the interval $[\widetilde \sigma_x, n]$ is lowered by $x/2$ and inserted at time $\sigma_{x/2}+2$, while the piece of $V$ on the interval $[\sigma_{x/2}+1, \widetilde \sigma_x-1]$ is reflected across $x/2$ and reattached at the end. More precisely, set
\begin{align*}
& H_{n,t}^{R,x}(V)_i:=V_i \quad \text{for}\quad   i\leq \sigma_{x/2},\\
& H_{n,t}^{R,x}(V)_{\sigma_{x/2}+1}:=x/2,\\
& H_{n,t}^{R,x}(V)_{\sigma_{x/2}+1+i}:=V_{\widetilde \sigma_x +i-1}-x/2 \quad \text{for}\quad  i=1,\dots,n+1-\widetilde \sigma_x,\\
& H_{n,t}^{R,x}(V)_{n+2-(\widetilde \sigma_x-\sigma_{x/2})+i}:=x-V_{\sigma_{x/2}+i} \quad \text{for} \quad
i=1,\dots,\widetilde \sigma_x-\sigma_{x/2}-1,\\
& H_{n,t}^{R,x}(V)_{n+2}:=0.
\end{align*}
We note that the sum of absolute increments of $H_{n,t}^{R,x}(V)$ equals that of $V$, and $H_{n,t}^{R,x}(V)$ is confined to $[0, R]$.
Therefore $H_{n,t}^{R,x}(V) \in \cA_{n+2,t+2}^R$. It remains to bound the number of pre-images of every $V\in \cA_{n+2,t+2}^R\cap H_{n,t}^{R,x}(\cD_{n,t}^{R,x})$ under $H_{n,t}^{R,x}$. Note that to undo $H_{n,t}^{R,x}(V)$, we only need to find the two times $\sigma_{x/2}+2$
and $\sigma_{x/2}+n+2-\widetilde \sigma_x$ at which the original segments of $V$ are glued together and $H_{n,t}^{R,x}(V)_i=0$. Since there are at most $n^2\leq t^2$ such choices, and combined with similar estimates for cases 2--4, we have
\begin{align}
 \widehat L_R(t,n)&= \sum_{x=1}^{R-1} \bE\Big[e^{\log(\frac{3}{2}) n-\lambda^*t} \, \ind_{\, \cD_{n,t}^{R,x}} (V)\Big]\\
\nonumber & \leq  4 (R-1)\,  t^2 \Big( 3^2 e^{-2 \log(\frac{3}{2})+2\lambda^*}\bE\Big[e^{\log(\frac{3}{2}) (n+2)-\lambda^*(t+2)} \, \ind_{\, \cA_{n+2,t+2}^{R}} (V)\Big]\Big)\\
\nonumber & \leq  C R\,  t^2 L_{R}(t+2,n+2,0),
\end{align}
which establishes \eqref{supnu} and hence the lemma.
\epr

\smallskip

As a corollary of Lemma \ref{Cbonudlast}, we have the following bound on the last ratio in \eqref{probeve}:
\be{nrdcor}
\frac{\widehat L_{R_{\gamma_L}}(L-(T_1+\cdots +T_{\gamma_L}))}{\bP_{R_{\gamma_L}}^*\big( T> L-(T_1+\cdots+T_{\gamma_L})\big)}\leq C\, L(L-(T_1+\cdots+T_{\gamma_L}))^2.
\ee

\section{Proof of Theorems \ref{unifscal} and \ref{unifscal2}}\label{sec5}
We will use the excursion decomposition developed in Section \ref{pw}, in particular, the representation in \eqref{probeve}.
First we show that for large $L$, a uniform prudent walk typically crosses its range at most $\log L$ times. Namely,

\begin{lemma}\label{finiteexc}
There exists $\delta>0$ such that
\be{excufi}
\lim_{L\to\infty} \bP_{{\rm unif},L}\big[\exists\, i\in \{\delta \log L,\dots,\gamma_L(\pi)\}\colon\, \cE_i(\pi)=1\big]=0.
\ee
\end{lemma}

Then we show that the total length of the first $\log L$ excursions grows less than a power of $\log L$.

\begin{lemma}\label{finitesteps}
For every $\delta>0$, there exists $\kappa>0$ such that
\be{excufi2}
\lim_{L\to \infty} \bP_{{\rm unif},L}\big[T_1(\pi)+\dots+T_{\delta \log L}(\pi)\geq  \kappa \, (\log L)^2 \big]=0.
\ee
\end{lemma}

Finally, we show that the last incomplete excursion of the walk typically has length at most $\log L$.

\begin{lemma}\label{lastExc}
There exists $\alpha>0$ such that
\be{excufi3}
\lim_{L\to \infty} \bP_{{\rm unif},L}\big[L-(T_1+\dots+T_{\gamma_L})\geq \alpha \log L \big]=0.
\ee
\end{lemma}

We prove Theorem \ref{unifscal} next using Lemmas \ref{finiteexc}--\ref{lastExc}, whose proof are postponed to Sections \ref{dec:excu}--\ref{dec:excu3}.

\subsection{Proof of Theorems \ref{unifscal} and \ref{unifscal2}}\label{Theo1}

Let $\delta,\kappa,\alpha>0$, and we define $\cG_L\subset \Omega_L$ by
$$
\cG_{L}:=\Big\{\cE_i=0\, \forall\, i\in \{\delta \log L,\dots, \gamma_L\}, \,
        T_1+\dots+T_{\delta \log L}\leq \kappa (\log L)^2, \  L-(T_1+\dots+T_{\gamma_L})\leq \alpha \log L \Big\} .
$$
By Lemmas \ref{finiteexc}--\ref{lastExc}, we can choose  $\delta, \kappa$ and $\alpha$ such that
$\lim_{L\to\infty}\bP_{\text{unif},L}\big( \cG_L \big)=1$.
\smallskip

We introduce a little more notation. Let $\cO :=\{\text{NE,NW,SE,SW}\}$ be the set of possible directions of a 2-sided prudent path. For $o\in \cO$ let $\Omega_L^{\, o}$ be the set of $L$-step 2-sided path with orientation $o$ (e.g.
$\Omega_L^{\text{NE}}=\Omega_L^+$). Pick $\pi \in \Omega_L$ and recall that the endpoint of each excursion of $\pi$ lies at one of the
4 corners (indexed in $\cO$) of the smallest rectangle containing the range of $\pi$ up to that endpoint.
Thus,  for $\pi \in \cG_L$, we denote by $\theta(\pi)\in \cO$ the corner at which the endpoint of the
$\delta \log L$-th excursion lies.
%Finally, for $t\in \N$, define a set of paths $\cI_t^{+}$ exactly as $\cI_t$ in \eqref{defIt}, except that the endpoint constraint $\pi_{t,2}=0$ is replaced by $\pi_{t,2}>0$.
\smallskip

%$m\geq \delta \log L$ and $o\in \cO$, let $\widetilde \Omega_m^o$ be the set of prudent paths $\pi \in \Omega_m$ which perform exactly $\delta \log L$ excursions in its range, and its end point $\pi_m$ lies at the corner (in direction $o$) of the smallest rectangle containing its range. Finally, for $t\in \N$, define a set of paths $\cI_t^{+}$ exactly as $\cI_t$ in \eqref{defIt}, except that the endpoint constraint $\pi_{t,2}=0$ is replaced by $\pi_{t,2}>0$.
%\smallskip

For a path $\pi\in \cG_L$, let $\sigma_1:=T_1+\dots+ T_{\delta\log L}$ be the length of the first $\delta\log L$ excursions, and
let $\sigma_2:=L-(T_1+\dots + T_{\gamma_L})$ be the length of the last incomplete excursion. Note that $(\pi_i)_{i=\sigma_1}^{L-\sigma_2}$ is a 2-sided prudent path of orientation  $\theta(\pi)$ because $\cE_i=0$ for $\delta \log L<i \leq  \gamma_L(\pi)$. Therefore, we can safely enlarge a bit $\cG_L$ into
$$
\widetilde \cG_{L}:=\Big\{(\pi_i)_{i=\sigma_1}^{L-\sigma_2} \in \Omega_{L-\sigma_1-\sigma_2}^{\theta(\pi)}, \,
        T_1+\dots+T_{\delta \log L}\leq \kappa (\log L)^2, \  L-(T_1+\dots+T_{\gamma_L})\leq \alpha \log L \Big\} .
$$
Note that  conditioned on $\pi\in \widetilde \cG_L$, $\sigma_1(\pi)=m$, $\sigma_2(\pi)=n$, and $\theta(\pi) = o$, the law of $(\pi_i)_{i=m}^{L-n}$ under $\bP_{\text{unif},L}$ (modulo translation and rotation) is exactly that of a uniform 2-sided prudent walk with total length $L-m-n$, for which we have proved the law of large numbers in Theorem \ref{scalingtwosided} and the invariance principle in Theorem \ref{scalingtwosided2}. Since $\bP_{\text{unif},L}\big( \widetilde \cG_L \big)\to 1$, we only need to consider $m\leq \kappa (\log L)^2$ and $n\leq \alpha \log L$. Since $m/\sqrt{L}, n/\sqrt{L}$ tend to $0$ uniformly as $L$ tends to infinity, $(\pi_i)_{i=1}^m$ and $(\pi_i)_{i=L-n}^{L}$ are negligible in the scaling limit, and hence Theorems \ref{unifscal} and \ref{unifscal2} follow from their counterparts for the uniform 2-sided prudent walk, with the direction $o$ distributed uniformly in $\cO$ by symmetry.
\qed

\subsection{Proof of Lemma \ref{finiteexc}}\label{dec:excu}

Let $M=M(L)$ be an increasing function of $L$  that will be specified later. We set
\be{fracretur}
\alpha_L:=\bP_{\text{unif},L}\big( \exists\, i\in [M,\gamma_L] \, s.t.\, \cE_i(\pi)=1\big)=\frac{\big|\{\pi\in \Omega_L: \,\exists\, i\in [M,\gamma_L]\, s.t.\, \cE_i(\pi) =1\} \big|}{|\Omega_L|}.
\ee
Multiply both the numerator and denominator by $2^{-L} e^{-\lambda^* L}$, we can then apply \eqref{probeve} together with \eqref{nrdcor} to obtain
\begin{align}
\alpha_L & \leq C L^3
\frac{\, \sum_{j=M}^L \bE^*\Big[\ind_{\{\cE_j=1\}}\,  \prod_{i=1}^{\gamma_L}  \frac{L_{R_{i-1}}(T_i,N_i,\cE_i)}{L^*_{R_{i-1}}(T_i,N_i,\cE_i)}
\Big]}{\bE^*\Big[\prod_{i=1}^{\gamma_L}  \frac{L_{R_{i-1}}(T_i,N_i,\cE_i)}{L^*_{R_{i-1}}(T_i,N_i,\cE_i)} \ind_{\{T_1+\dots+T_{\gamma_L}=L\}}
\Big]} \nonumber \\
& \leq C L^3
\frac{\,\sum_{j\geq M} \bE^*\Big[ \ind_{\{T_j \geq \frac{j}{2}\}}  \, \prod_{i=1}^{\gamma_L}  \frac{L_{R_{i-1}}(T_i,N_i,\cE_i)}{L^*_{R_{i-1}}(T_i,N_i,\cE_i)}
\Big]}{\bE^*\Big[\prod_{i=1}^{\gamma_L}  \frac{L_{R_{i-1}}(T_i, N_i,\cE_i)}{L^*_{R_{i-1}}(T_i, N_i,\cE_i)} \ind_{\{T_1+\dots+T_{\gamma_L}=L\}}
\Big]}:= CL^3 \frac{\Psi_1(L,M)}{D_L}, \label{fracreturrr}
\end{align}
where we used that $\cE_i=1$ only if $T_i\geq 1+R_{i-1}$, and  $R_i\geq \frac{i-1}{2}$ for every $i\in \N$ (cf.\ Section \ref{decppe}).
Lemma \ref{finiteexc} then follows immediately from \eqref{fracreturrr} and Claims \ref{c2} and \ref{c1} below.

\begin{claim}\label{c2}
There exist $c_1,c_2>0$ such that
$\Psi_1(L,M) \leq c_1\,  e^{-c_2 M}$ for every $M\in \N$ and $L\geq M$.
\end{claim}

\begin{claim}\label{c1}
There exists $c_3>0$ such that $D_L\geq c_3$ for every $L\in \N$.
\end{claim}

\medskip

{\bf Proof of Claim \ref{c2}.} Recall from Section \ref{sec:sampling} how $(T_i,N_i,\cE_i)_{i\geq 1}$ is constructed from the i.i.d.\ sequence $(\widetilde V_i, \widetilde T_i,\widetilde N_i)_{i\geq 1}$ with law $\bP^*$, with $\tilde T_i\geq T_i\, \forall\, i\in \N$. We first state and prove a key lemma.

\begin{lemma}\label{lemmaKey}
Let $L\in \N$, and let $\Phi:\R^L_+\to \R_+$ be any function that is non-decreasing in each of its $L$ arguments. Then there exists $c>0$ independent of $L$ and $\Phi$, such that
\be{eqKey1}
\bE^*\bigg[\!\Phi(T_1,\dots, T_L) \prod_{i=1}^{\gamma_L}  \frac{L_{R_{i-1}}(T_i,N_i,\cE_i)}{L^*_{R_{i-1}}(T_i,N_i,\cE_i)} \bigg]
\leq
\bE^*\bigg[\Phi(\tilde T_1,\dots,\tilde T_L) \prod_{i=1}^{L}  \big(1 + c  \tilde T_i\,  \ind_{\{\tilde T_i\geq \frac{i-1}{2}\}}\big) \bigg].
\ee
\end{lemma}
\bpr
For $n\in \N$, let $\cF_n$ be the $\sigma$-algebra generated by $(\widetilde T_i, T_i,N_i,\cE_i)_{i\leq n}$. For ease of notation, let $A_L$ denote the l.h.s.\ of \eqref{eqKey1}. Note that
\begin{align}\label{LL*0}
A_L & \leq \bE^*\bigg[ \Phi( T_1,\dots, T_L) \prod_{i=1}^{L}  \max \bigg\{ \frac{L_{R_{i-1}}(T_i,N_i,\cE_i)}{L^*_{R_{i-1}}(T_i,N_i,\cE_i)}, 1\bigg\} \bigg] \\
& = \bE^*\bigg[\prod_{i=1}^{L-1}  \max \Big\{ \frac{L_{R_{i-1}}(T_i,N_i,\cE_i)}{L^*_{R_{i-1}}(T_i,N_i,\cE_i)}, 1\Big\}  \, \, H_L\bigg], \label{LL*02}
\end{align}
with
\begin{equation} \label{defHL}
\begin{split}
H_L & := \bE^*\bigg[\Phi( T_1,\dots, T_L)   \max \Big\{ \frac{L_{R_{L-1}}(T_L,N_L,\cE_L)}{L^*_{R_{L-1}}(T_L,N_L,\cE_L)}, 1\Big\}\Big| \cF_{L-1}\bigg]  \\
&\ = \sum_t \Phi(T_1,\dots,T_{L-1}, t) \sum_{n\leq t}\sum_{\epsilon=0,1} \max \big\{ {L_{R_{L-1}}(t,n,\gep)},  L^*_{R_{L-1}}(t,n,\gep) \big\}.
\end{split}
\end{equation}
When $t<R_{L-1}$, we have ${L_{R_{L-1}}(t,n,1)}= L^*_{R_{L-1}}(t,n,1)=0$, and ${L_{R_{L-1}}(t,n,0)}= L^*_{R_{L-1}}(t,n,0)$ by \eqref{L0=L*0}, so that
\begin{equation}
\sum_{n\leq t}\sum_{\epsilon=0,1} \max \big\{ {L_{R_{L-1}}(t,n,\gep)},  L^*_{R_{L-1}}(t,n,\gep) \big\} = L^*_{R_{L-1}}(t).
\end{equation}
When $t\geq R_{L-1}$, we have
\begin{equation}
\begin{split}
& \sum_{n\leq t}\sum_{\epsilon=0,1} \max \big\{ {L_{R_{L-1}}(t,n,\gep)},  L^*_{R_{L-1}}(t,n,\gep) \big\}  \\
& \quad \leq \  \sum_{n\leq t}\sum_{\epsilon=0,1} ({L_{R_{L-1}}(t,n,\gep)} + L^*_{R_{L-1}}(t,n,\gep))
= L_{R_{L-1}}(t) + L^*_{R_{L-1}}(t) \leq (1+ct) L^*_{R_{L-1}}(t),
\end{split}
\end{equation}
where we applied Lemma \ref{Cbonud}. Therefore we have
\begin{equation}\label{rtg}
\begin{split}
H_L&\leq \sum_{t} \Phi(T_1,\dots, T_{L-1},t) \big(1+ct \ind_{\{ t\geq R_{L-1}\}}\big) L^*_{R_{L-1}}(t),
\\
&=
\bE^*\left[ \Phi(T_1,\dots,T_{L-1}, T_L) \big( 1+cT_L \ind_{\{ T_L\geq R_{L-1}\}} \big) \,  \big|\, \cF_{L-1}\right].
\end{split}
\end{equation}
Since $R_{L-1}\geq \frac{L-1}{2}$  and $\widetilde T_L\geq T_L$, we can replace $R_{L-1}$ by $\frac{L-1}{2}$ and $T_L$ by $\tilde T_L$ in the r.h.s. of \eqref{rtg}. Moreover, note that $\tilde T_L$ does not depend on $R_{L-1}$, and hence we can plug \eqref{rtg} into \eqref{LL*0} to obtain
\be{LL*2}
\nonumber \begin{split}
A_L \leq
\bE^*\bigg[ \bigg( \prod_{i=1}^{L-1}  \max \big\{ \frac{L_{R_{i-1}}(T_i,N_i,\cE_i)}{L^*_{R_{i-1}}(T_i,N_i,\cE_i)}, 1\big\} \bigg)
\Phi(T_1,\dots, T_{L-1},\tilde T_L)\big(1+ c \tilde T_L\,  \ind_{\{\tilde T_L\geq \frac{L-1}{2}\}}\big)\bigg].
\end{split}
\ee
We can now iterate the argument to deduce \eqref{eqKey1}.
\epr

To prove Claim \ref{c2}, we now apply Lemma \ref{lemmaKey} with $\Phi(t_1,\dots,t_L)=\ind_{\{t_j\geq \frac{j}{2} \}}$
for $j\geq M$ to obtain
\begin{align}
\Psi_1(L,M) & \leq \sum_{j\geq M} \bE^*\bigg[ \ind_{\{\tilde T_j\geq \frac{j}{2} \}} \prod_{i=1}^{L}  \big(1 + c \tilde T_i\,  \ind_{\{\tilde T_i\geq \frac{i-1}{2}\}}\big) \bigg] \nonumber\\
& \leq \sum_{j\geq M} \bE^*\bigg[ (1+c\tilde T_j) \ind_{\{\tilde T_j\geq \frac{j}{2} \}} \prod_{i\neq j\leq L}  \Big(1 + c\,  \tilde T_i\,  \ind_{\{\tilde T_i\geq \frac{i-1}{2}\}}\Big)
\bigg] \nonumber \\
& = \sum_{j\geq M} \bE^*\bigg[ (1+c\tilde T_1) \ind_{\{\tilde T_1\geq \frac{j}{2} \}}\bigg] \prod_{i\neq j\leq L} \Big(1 + c \bE^*[\tilde T_1 \ind_{\{\tilde T_1\geq \frac{i-1}{2}\}}]\Big)
\end{align}
Since $\widetilde T_1$ has exponential tail under $\bP^*$ (cf. Remark \ref{remKK}), there exist $C_1, C_2>0$ such that
\be{Estar}
\bP^*(\tilde T_1\geq l) \leq \bE^*\big[\tilde T_1 \ind_{\{\tilde T_1\geq  \ell\}}\big] \leq C_1\,  e^{-C_2 \ell} \qquad \mbox{for all } \ell \in \N.
\ee
This implies that
\begin{equation}\label{defA3}
\Psi_1(L,M) \leq (1+c) \sum_{j\geq M} C_1 e^{-C_2\,  \frac{j}{2}} \prod_{i=1}^\infty \Big(1+c\,  C_1 e^{-C_2\,  \frac{i-1}{2}}\Big)
\leq c_1 e^{-c_2\, M},
\end{equation}
which concludes the proof of Claim \ref{c2}.
\qed
\medskip

{\bf Proof of Claim \ref{c1}}
The claim is essentially a consequence of the renewal theorem. Note that by construction, we have $R_0=0$, $\cE_1=1$, and $L_0(T_1, N_1, 1)= L^*_0(T_1, N_1, 1)$, and when $R_{i-1}\geq 1$ and $T_i=1$, or when $T_i<R_{i-1}$, we must have $\cE_i=0$ and $L_{R_{i-1}}(T_i, N_i,0)=L^*_{R_{i-1}}(T_i, N_i,0)$. Therefore, with $A>0$ to be chosen later, we can bound
\be{loBdeno}
\begin{split}
M_L:=\ &{\bE^*\bigg[\prod_{i=1}^{\gamma_L}  \frac{L_{R_{i-1}}(T_i, N_i,\cE_i)}{L^*_{R_{i-1}}(T_i, N_i,\cE_i)} \ind_{\{T_1+\dots+T_{\gamma_L}=L\}}
\bigg]} \\
\geq\ &
\bP^*\big( T_1+\dots+T_{\gamma_L}=L,\ T_i=1\, \forall\,  i\in [1, A], \ T_j < R_{j-1} \, \forall\, j\in [A+1, \gamma_L]
\big).
\end{split}
\ee
Recall that $(T_i, N_i, \cE_i)_{i\in\N}$ is constructed from $(\widetilde V_i, \widetilde T_i, \widetilde N_i)_{i\in\N}$ with law $\bP^*$ such that $\tilde T_i\geq T_i$ a.s., and when $\widetilde T_i=1$ or $\widetilde T_i \leq R_{i-1}$, we have $T_i=\tilde T_i$ (cf.\ Section \ref{sec:sampling}). Since $R_1, R_2\geq 1$ and $R_i\geq \frac{i-1}{2}$ for $i\geq 3$, we can bound the r.h.s.\ of \eqref{loBdeno} by
\begin{align}\label{thuiboun}
\nonumber M_L&\geq \bP^*\big(\tilde T_1+\dots+ \tilde T_{\tilde \gamma_L}=L,\ \tilde T_i=1\ \forall \, i\in [1, A], \
\tilde T_j < \tfrac{j-1}{2}\ \forall\, j>A\big),\\
 &=\bP^*(\tilde T_1=1)^{A} \   \bP^*\Big(\tilde T_1+\dots+ \tilde T_{\, \tilde \gamma_{L-A}}=L-A,   \ \tilde T_i < \tfrac{A+i-1}{2}\ \forall\, i\geq 1 \Big),
\end{align}
where $\tilde \gamma_L$ is the counterpart of $\gamma_L$ for  $(\widetilde T_i)_{i\in \N}$ (recall \eqref{gammaLV}). Since $(\tilde T_j)_{j\in\N}$ is i.i.d. with exponential tail, we may pick $A\in \N$ large enough such that
\be{guar1}
\bP^*\Big(\tilde T_i < \frac{A+i-1}{2}\ \forall\, i\geq 1 \Big) \geq 1-\frac{1}{4\bE^*[\tilde T_1]}.
\ee
Having chosen $A$, the renewal theorem then ensures that there exists $L_0\in \N$ such that
\be{guar2}
\bP^*(\tilde T_1+\dots+ \tilde T_{\, \tilde \gamma_{L-A}}=L-A)\geq \frac{1}{2\bE^*[\tilde T_1]} \quad \forall\  L>L_0.
\ee
Combining \eqref{guar1} and \eqref{guar2} then shows that the r.h.s.\ in \eqref{thuiboun} is bounded from below by a positive constant uniformly in $L\geq L_0$. The proof is then complete.
\qed

\subsection{Proof of Lemma \ref{finitesteps}}

The proof is similar to that of Lemma \ref{finiteexc}. Let $\delta>0$ and $\kappa>0$ and set
$$
\beta_L:=\bP_{\text{unif},L}\big( T_1(\pi)+\dots+T_{\delta \log L}(\pi) \geq \kappa (\log L)^2 \big)=
\frac{|\{\pi\in \Omega_L:  T_1+\cdots+T_{\delta \log L} \geq \kappa (\log L)^2\}|}{|\Omega_L| }.
$$
Since $T_1(\pi)+\cdots +T_{\delta \log L} \geq \kappa (\log L)^2$ implies that $T_i\geq \kappa\delta^{-1} \log L$ for some
$1\leq i\leq \delta \log L$, similar to \eqref{fracreturrr}, we have
\be{fracretur22}
\beta_L\leq C L^3 \,
\frac{\sum_{j=1}^{\delta \log L} \bE^*\Big[ \ind_{\{T_j\geq {\kappa \delta^{-1} \log L}\}} \prod_{i=1}^{\gamma_L}  \frac{L_{R_{i-1}}(T_i, N_i,\cE_i)}{L^*_{R_{i-1}}(T_i, N_i, \cE_i)} %\ind_{\{T_1+\dots+T_{\gamma_L}=L\}}
\Big]}{\bE^*\Big[\prod_{i=1}^{\gamma_L}  \frac{L_{R_{i-1}}(T_i, N_i, \cE_i)}{L^*_{R_{i-1}}(T_i, N_i,\cE_i)} \ind_{\{T_1+\dots+T_{\gamma_L}=L\}}
\Big]}:= C L^3 \frac{\Psi_2(L)}{D_L}.
\ee
By Claim \ref{c1}, $D_L$ is bounded away from $0$ uniformly in $L$.  Using Lemma \ref{lemmaKey} and \eqref{Estar}, we obtain
\begin{align}\label{B1}
\nonumber \Psi_2(L)
&\leq
\sum_{j=1}^{ \delta \log L}
{\bE^*\bigg[ \ind_{\{\tilde T_j\geq {\kappa \delta^{-1} \log L}\}} \prod_{i=1}^{L}  \left(1+ c \tilde T_i\ind_{\{ \tilde T_i \geq \frac{i-1}{2}\}}\right)
\bigg]} \nonumber\\
&\leq
(\delta \log L)\,  \bE^*\Big[(1+c\tilde T_1) \ind_{\{\tilde T_1\geq \kappa {\delta^{-1} \log L}\}}\Big] \prod_{i=1}^{\infty}
\Big(1+ c\bE^*[\tilde T_1\ind_{\{ \tilde T_1 \geq \frac{i-1}{2}\}}]\Big) \nonumber\\
&\leq (\delta \log L) c_1 e^{-c_2 \kappa {\delta^{-1} \log L}},
\end{align}
which tends to $0$ as $L$ tends to infinity if $\kappa$ is chosen large enough.
\qed

\subsection{Proof of Lemma \ref{lastExc}}\label{dec:excu3}
As in the proof of Lemmas \ref{finiteexc} and \ref{finitesteps}, given $\delta>0$, we set
$$
\rho_L:=\bP_{\text{unif},L}\big( L-(T_1(\pi)+\dots+T_{\gamma_L}(\pi)) \geq \alpha \log L \big)=
\frac{|\{\pi\in\Omega_L: ( L-(T_1+\cdots+T_{\gamma_L}) \geq \alpha \log L)\}|}{|\Omega_L| }.
$$
Similar to \eqref{fracreturrr}, we have
\be{thu}
\rho_L \leq
CL^3\, \frac{\bE^*\bigg[ \ind_{\{L-(T_1+\dots+T_{\gamma_L}) \geq \alpha \log L \}} \prod_{i=1}^{\gamma_L}  \frac{L_{R_{i-1}}(T_i, N_i,\cE_i)}{L^*_{R_{i-1}}(T_i, N_i, \cE_i)} %\ind_{\{T_1+\dots+T_{\gamma_L}=L\}}
\bigg]}{\bE^*\bigg[\prod_{i=1}^{\gamma_L}  \frac{L_{R_{i-1}}(T_i, N_i, \cE_i)}{L^*_{R_{i-1}}(T_i, N_i,\cE_i)} \ind_{\{T_1+\dots+T_{\gamma_L}=L\}}
\bigg]}=:C L^3 \frac{\Psi_3(L)}{D_L}.
\ee
By Claim \ref{c1}, $D_L$ is bounded away from $0$ uniformly in $L$.
Since $L-(T_1+\dots+T_{\gamma_L})\geq \alpha \log L$ implies $\max\{T_1,\dots,T_L\}\geq \alpha \log L$, again by Lemma \ref{lemmaKey} and \eqref{Estar}, we have
\begin{align}\label{trio}
\nonumber \Psi_3(L)&\leq \, \bE^*\bigg[ \ind_{\{ \max\{T_1,\dots,T_L\} \geq \alpha  \log L \}} \prod_{i=1}^{\gamma_L}  \frac{L_{R_{i-1}}(T_i, N_i,\cE_i)}{L^*_{R_{i-1}}(T_i, N_i, \cE_i)}
\bigg]\\
&\leq
 \bE^*\bigg[ \ind_{\{ \max\{\tilde T_1,\dots,\tilde T_L\} \geq \alpha \log L \}}  \prod_{i=1}^L  \big(1 + c\,  \tilde T_i\,  \ind_{\{\tilde T_i\geq \frac{i-1}{2}\}}\big)\bigg]  \nonumber \\
& \leq L \bE^*\Big[(1+c\tilde T_1) \ind_{\{\tilde T_1\geq \alpha \log L\}}\Big] \prod_{i=1}^{\infty}
\Big(1+ c\bE^*[\tilde T_1\ind_{\{ \tilde T_1 \geq \frac{i-1}{2}\}}]\Big) \leq c_1 L e^{-c_2 \alpha \log L},
\end{align}
which tends to $0$ as $L$ tends to infinity if $\alpha$ is chosen large enough.
\qed

%\subsection{Link between the uniform and the kinetic prudent walk}\label{link}
%%Let $\Omega_L$ be the set of $L$-step n-dimensional prudent paths. Then, we denote by
%Let $\bP_{\text{kin},L}$ the law of the kinetic prudent walk of size $L$, cf. \eqref{linkpruki}, and by $\bP_{\text{unif},L}$ its uniform counterpart.
%
%Since every $\pi\in \Omega_L$ satisfies $\bP_{\text{unif},L}(\pi) =1/ |\Omega_L|$, we deduce from \eqref{linkpruki} that
%\begin{align}\label{linkpruke}
%\bP_{\text{unif},L}(\pi)&=\frac{1}{\frac{1}{4}\frac{1}{2^L} |\Omega_L|} \big(\tfrac{3}{2}\big)^{\cH(\pi)+\cW(\pi)}   \bP_{\text{kin},L}(\pi)
%\end{align}
%and consequently
%\begin{align}\label{linkpruka}
%\bP_{\text{unif},L}(\pi)&= \frac{ \big(\tfrac{3}{2}\big)^{\cH(\pi)+\cW(\pi)}}{\bE_{\text{kin},L}\Big[\big(\tfrac{3}{2}\big)^{\cH(\pi)+\cW(\pi)}\Big]}   \bP_{\text{kin},L}(\pi).
%\end{align}

\bibliographystyle{imsart-nameyear}
\bibliography{cnp}

\end{document}